\documentclass{article}

\usepackage{geometry}                
\usepackage{graphicx}

\usepackage{enumerate}

\usepackage{amsmath}
\usepackage{amssymb}
\usepackage{amsthm}
\usepackage{url}
\usepackage{color}

\usepackage{natbib}

\usepackage{tikz-cd}

\usepackage{float}
\floatstyle{boxed} 
\restylefloat{figure}

\usepackage{chngcntr}
\counterwithin{figure}{section}

\usepackage{pifont}
\usepackage{stmaryrd}
\usepackage{dsfont}
\usepackage{textcomp}
\usepackage{verbatim}
\usepackage{accents}
\usepackage{wasysym}
\usepackage{csquotes}

\newtheorem{thm}{Theorem}[section]
\newtheorem{prop}[thm]{Proposition}
\newtheorem{lemma}[thm]{Lemma}
\newtheorem*{mainthm*}{Main Theorem}
\newtheorem*{mainlemma*}{Main Lemma}
\newtheorem{cor}[thm]{Corollary}

\newtheorem*{conj*}{Conjecture}

\theoremstyle{definition}

\newtheorem{dfn}[thm]{Definition}

\newtheorem{constr}[thm]{Construction}

\newtheorem{ax}[thm]{Axiom}
\newtheorem{sys}[thm]{System}

\newtheorem*{que*}{Question}
\newtheorem*{mque*}{Main Question}

\newtheorem*{rem*}{Remark}
\newtheorem*{Copernican*}{Copernican Principle}

\newcommand{\lang}{\mathcal{L}}

\newcommand{\sent}{\mathrm{Sent}}

\newcommand{\crsmb}{\mathbf{crsm}}

\newcommand{\SSy}{\mathrm{SSy}}

\newcommand{\udot}[1]{\text{\d{${#1}$}}}

\newcommand{\gquote}[1]{\ulcorner #1 \urcorner}

\newcommand{\negd}{\mathop{\udot{\neg}}}
\newcommand{\wedged}{\mathrel{\udot{\wedge}}}

\newcommand{\foralld}{\mathop{\udot{\forall}}}

\newcommand{\df}{\mathrm{df}}

\newcommand{\var}{\mathrm{Var}}
\newcommand{\VA}{\mathrm{VA}}

\newcommand{\Con}{\mathrm{Con}}
\newcommand{\R}{\mathrm{R}}
\newcommand{\Prv}{\mathrm{Pr}}
\newcommand{\Iso}{\mathrm{Iso}}
\newcommand{\restr}{\hspace{-4pt}\restriction}
\newcommand{\restrp}{\hspace{-4pt}\restriction + \hspace{2pt}}
\newcommand{\Modrm}{\mathrm{Mod}}
\newcommand{\Unirm}{\mathrm{Uni}}
\newcommand{\Trm}{\mathrm{T}}

\newcommand{\Sat}{\mathsf{Sat}}
\newcommand{\Tr}{\mathsf{Tr}}
\newcommand{\Mod}{\mathsf{Mod}}

\newcommand{\Uni}{\mathsf{Uni}}
\newcommand{\Unib}{\mathbf{Uni}}

\newcommand{\Tb}{\mathbf{T}}
\newcommand{\self}{\mathsf{self}}

\newcommand{\crsm}{\text{crsm}}
\newcommand{\Rev}{\mathrm{Rev}}

\newcommand{\dt}{\hspace{2pt}}

\newcommand{\PA}{\mathsf{PA}}
\newcommand{\ZF}{\mathsf{ZF}}
\newcommand{\ZFC}{\mathsf{ZFC}}

\newcommand{\CT}{\mathsf{CT}}
\newcommand{\FS}{\mathsf{FS}}
\newcommand{\CM}{\mathsf{CM}}

\newcommand{\GR}{\mathsf{GR}}
\newcommand{\SP}{\mathsf{SP}}
\newcommand{\MS}{\mathsf{MS}}
\newcommand{\GL}{\mathsf{GL}}
\newcommand{\HM}{\mathsf{HM}}

\newcommand{\ConTh}{\mathsf{Con}}
\newcommand{\RTh}{\mathsf{R}}

\newcommand{\NEC}{\mathsf{NEC}}
\newcommand{\CONEC}{\mathsf{CONEC}}
\newcommand{\Sep}{\mathsf{Sep}}
\newcommand{\Rep}{\mathsf{Rep}}

\title{The Copernican Multiverse of Sets}
\author{
Paul K. Gorbow and Graham E. Leigh \\
\\
	\small University of Gothenburg \\
	\small Department of Philosophy, Linguistics, and Theory of Science \\
	\small Box 200, 405 30 G\"OTEBORG, Sweden \\
}
\date{}                                           %

\begin{document}

\maketitle

\begin{abstract}
We develop an untyped framework for the multiverse of set theory. $\ZF$ is extended with semantically motivated axioms utilizing the new symbols $\Uni(\mathcal{U})$ and $\Mod(\mathcal{U, \sigma})$, expressing that $\mathcal{U}$ is a universe and that $\sigma$ is true in the universe $\mathcal{U}$, respectively. Here $\sigma$ ranges over the augmented language, leading to liar-style phenomena that are analysed. The framework is both compatible with a broad range of multiverse conceptions and suggests its own philosophically and semantically motivated multiverse principles. In particular, the framework is closely linked with a deductive rule of Necessitation expressing that the multiverse theory can only prove statements that it also proves to hold in all universes. We argue that this may be philosophically thought of as a {\em Copernican principle} that the background theory does not hold a privileged position over the theories of its internal universes. 

Our main mathematical result is a lemma encapsulating a technique for locally interpreting a wide variety of extensions of our basic framework  in more familiar theories. We apply this to show, for a range of such semantically motivated extensions, that their consistency strength is at most slightly above that of the base theory $\ZF$, and thus not seriously limiting to the diversity of the set-theoretic multiverse. We end with case studies applying the framework to two multiverse conceptions of set theory: arithmetic absoluteness and Joel D. Hamkins' multiverse theory.
\end{abstract}

\section{Introduction}

$\ZF$ set theory serves as a foundation for mathematics, but has also turned out to be interesting in itself as a field of mathematical study. Much of the interest lies in that it raises questions that are not only undecidable, but also lacking clear-cut intuitive answers and demanding deep mathematical developments. The continuum hypothesis is a primary historical example: It seems implausible to reach a consensus on affirming or denying it, and it motivated two techniques central to set theory: the inner model and forcing constructions. It is natural to view these techniques as enabling constructions of set-theoretic universes from other set-theoretic universes, thus taking a multiverse view of the subject matter of set theory, rather than adopting the universe view that there is a single absolute universe of sets. In the words of Hamkins, an advocate of the multiverse view (\citeyear[p. 418]{Ham12}): 

\begin{displayquote}
A large part of set theory over the past half-century has been about constructing as many different models of set theory as possible /.../ Would you like to live in a universe where CH holds, but $\Diamond$ fails? Or
where $2^{\aleph_n} = \aleph_{n+2}$ for every natural number $n$? Would you like to have rigid Suslin trees?
Would you like every Aronszajn tree to be special? Do you want a weakly compact cardinal $\kappa$ for which $\Diamond_\kappa(\mathrm{REG})$ fails? Set theorists build models to order. 
\end{displayquote}

Hamkins follows this perspective on set-theoretic practice with his argument for adopting the multiverse view:

\begin{displayquote}
This abundance of set-theoretic possibilities poses a serious difficulty for the universe view, for if one holds that there is a single absolute background concept of set, then one
must explain or explain away as imaginary all of the alternative universes that set theorists seem to have constructed. This seems a difficult task, for we have a robust experience in those worlds, and they appear fully set theoretic to us. The multiverse view, in contrast, explains this experience by embracing them as real, filling out the vision hinted at in our mathematical experience, that there is an abundance of set-theoretic worlds into which our mathematical tools have allowed us to glimpse.
\end{displayquote}

Methodologically, it makes sense to represent the universes of sets as models of a theory of sets, as that makes them accessible to the well-developed techniques of model theory. This raises two questions: 
\begin{enumerate}
\item Which set theory?
\item Which models of that set theory?
\end{enumerate}

In answer to the first question the authors have decided on limiting scope to $\ZF$. This is the most utilized set theory, established in the mathematical community as a robust foundation of mathematics. All of the results of this paper go through for extensions of $\ZF$ (axiom of choice, large cardinals, $\dots$). Although one development of this paper uses full $\ZF$,\footnote{The  proof of Theorem \ref{Thm: rec sat iso} requires full $\ZF$. This is used in the proof of Theorem \ref{Thm: CM + NEC + self-iso}, a consistency result.} the authors conjecture that for most of the results much weaker fragments suffice.

On the second question this paper takes a liberal approach. In particular, scope is not limited to well-founded models. The framework is intended to be applicable to a wide range of multiverse conceptions, for example the conceptions that every universe is of the form $V_\alpha$, that every universe is well-founded, that every universe is countable and recursively saturated (and therefore ill-founded), or that every model is a universe. 

Why consider ill-founded universes? There is a sense in which a model $\mathcal{M}$ of set theory can be situated as an element in two different models $\mathcal{N}_0$ and $\mathcal{N}_1$ of set theory, such that $\mathcal{N}_0$ satisfies that $\mathcal{M}$ is well-founded while $\mathcal{N}_1$ satisfies that $\mathcal{M}$ is ill-founded.\footnote{Using Definition \ref{Dfn: externalization}, the precise sense is that $\mathcal{M} = \mathcal{A}_{\mathcal{N}_0} = \mathcal{B}_{\mathcal{N}_1}$, for some $\mathcal{A} \in \mathcal{N}_0$ and $\mathcal{B} \in \mathcal{N}_1$, such that $\mathcal{N}_0$ satisfies that $\mathcal{A}$ is well-founded while $\mathcal{N}_1$ satisfies that $\mathcal{B}$ is ill-founded.} Thus, we may think of the property of well-foundedness as depending on the set-theoretic background. In Hamkins's multiverse conception this is a key feature \cite[p. 438-9]{Ham12}:

\begin{displayquote}
The concept of well-foundedness [\dots] depends on the set-theoretic background, for different models of set theory can disagree on whether a structure is well-founded. [\dots] Indeed, every set-theoretic argument can take place in a model, which from the inside appears to be totally fine, but actually, the model is seen to be ill-founded from the external perspective of another, better model. Under the universe view, this problem terminates in the absolute set-theoretic background universe, which provides an accompanying absolute standard of well-foundedness. But the multiverse view allows for many different set-theoretic backgrounds, with varying concepts of the well-founded, and there seems to be no reason to support an absolute notion of well-foundedness.
\end{displayquote}

To approach the set-theoretic multiverse mathematically, we need a foundational theory to situate the universes in. Just as the foundational background theory of $\ZF$ is useful for studying groups and topological spaces, it is useful for studying set-theoretic universes. So we find ourselves in a situation of studying models of $\ZF$ from the background theory $\ZF$. The multiverse theorist may extend the background theory of $\ZF$ to a multiverse theory (in an expanded language), with axioms specifying properties of the multiverse. In such a background multiverse theory, it is natural to consider the universes as themselves being models of multiverse theories, having their own internal universes, and so on. This raises an important question:

\begin{mque*}\label{Qu: Main}
What is the relationship between the external universe of the background multiverse theory, and the universes internal to the background theory? Similarly, what is the relationship between each universe and the universes within that universe?
\end{mque*}

We shall investigate several responses to the Main Question. Most fundamentally, the authors propose that the background multiverse theory obeys the following principle:

\begin{Copernican*}
The background theory of the multiverse should not have a privileged position compared to the multiverse theories of the internal universes; specifically, if the background multiverse theory proves a statement, it should also prove that holds in all universes.
\end{Copernican*}

We have an analogy with the heliocentric model of the solar system: Earth corresponds to the universe of the background multiverse theory, as the basis for our point of view. The geocentric model gives earth a privileged central position as an absolute reference point, while the heliocentric model puts earth on a par with all of the planets. So the heliocentric model differs from the geocentric model in that it obeys the principle that for any appropriately fundamental assumption we make about our point of view, we are committed to making the same assumption about every other plausible point of view. Similarly, in the context of set theory, the authors propose the Copernican Principle as the constraint that for every assumption introduced by a multiverse theorist, s/he is committed to that it holds from the vantage point of an arbitrary universe of sets. The name is borrowed from a modern principle in physics, which Peacock states as ``that humans are not privileged observers''. Peacock applies the principle arguing ``if the universe appears isotropic about our position, it would also appear isotropic to observers in other galaxies'' (\citeyear[p. 66]{Pea98}). So for the physicist, the principle is a conceivably falsifiable statement about the uniformity of the physical universe; while for the theorist of the multiverse of sets, it is an a priori postulate. Below we explicate a formal deductive rule, $\NEC$, expressing this principle.

To approach the Main Question, we require a framework that makes sense of the notion of truth-in-a-universe. If the universes are mere models of $\ZF$, then the usual satisfaction-relation expressed in the language of set theory suffices. But as soon as we consider each universe to contain a multiverse in its own right, it is more natural to consider the universes as structures in the language of the multiverse theory.

The main contribution of this paper is an untyped framework for handling the notion of truth-in-a-$\ZF$-universe. It is intended to be applicable to just about any theory of the multiverse of sets. A primitive predicate $\Uni(\mathcal{U})$ is introduced to express that $\mathcal{U}$ is a universe, and the primitive relation $\Mod(\mathcal{U}, \sigma)$ is introduced to express that the $\lang_{\Uni, \Mod}$-statement $\sigma$ is true in the universe $\mathcal{U}$, where $\lang_{\Uni, \Mod}$ is the language of set theory, $\lang$, augmented with the symbols $\Uni$ and $\Mod$. The multiverse theories considered in this paper are expressed in $\lang_{\Uni, \Mod}$.

\subsection{Multiverse principles}\label{Subsec: Multiverse principles}

Now that the language $\lang_{\Uni, \Mod}$ and the intuitive intended meaning of its symbols has been briefly explained, the next task is to give natural and useful axioms for $\Uni$ and $\Mod$. Axioms for $\Uni$ specify what universes comprise the multiverse, what closure properties it satisfies, etc.; for example the multiverse axioms of \cite{Ham12}. In this paper we are focused on semantically motivated axioms, meant to be applicable to a wide range of multiverse conceptions. The application of this framework to Hamkins's multiverse is discussed in \S \ref{Subsec: Hamkins's multiverse}. Since $\Mod$ is an untyped semantic relation, it is not surprising that it is exposed to liar-style phenomena. Our Theorem \ref{Thm: liar} shows, e.g., that the schema $\big( \forall \mathcal{U} \dt (\Uni(\mathcal{U}) \rightarrow \Mod(\mathcal{U}, \gquote{\sigma}) \big) \rightarrow \sigma$, over $\lang_{\Uni, \Mod}$-statements $\sigma$, expressing that whatever holds in every universe also holds in the background universe, is inconsistent with a natural and mild theory in $\lang_{\Uni, \Mod}$. However, this contradiction is not derivable when this schema is restricted to $\lang$.

The basic theory introduced is called $\CM^-$ (Compositional satisfaction for the Multiverse).\footnote{The precise specification of $\CM^-$ is given in System \ref{Sys: CM}.} $\CM^-$ is formed by adding compositional semantically motivated $\lang_{\Uni, \Mod}$-axioms to the background theory $\ZF$, for each logical connective and quantifier, and by extending the Separation and Replacement schemas of $\ZF$ to $\lang_{\Uni, \Mod}$. For example, the compositional axiom for $\wedge$ is
\[
\text{if $\theta \in \lang_{\Uni, \Mod}$ is the conjunction of $\phi$ and $\psi$, then $\Mod(\mathcal{U}, \theta) \iff \Mod(\mathcal{U}, \phi) \wedge \Mod(\mathcal{U}, \psi)$.}
\] 
These are also called the {\em Tarskian laws of satisfaction}. The analogue in the present framework of the well-known Tarskian schema $\Tr(\gquote{\sigma}) \leftrightarrow \sigma$ (for $\sigma \in \lang$) would say roughly that $\sigma$ is true in every universe if, and only if, it holds in the background universe. So this would say that the multiverse is not very diverse, and certainly not closed under forcing, for example. But analogues of $\Tr(\gquote{\sigma}) \rightarrow \sigma$ (for $\sigma \in \lang$) and the rule of Necessitation $\vdash \sigma \Rightarrow \dt \vdash \Tr(\gquote{\sigma})$ (for $\sigma \in \lang_{\Uni, \Mod}$) are highly relevant. 

In $\CM^-$ we can prove the soundness principle that the set of statements true in any particular universe is deductively closed. $\CM$ is $\CM^-$ plus an axiom called $\textnormal{\sf Multiverse}_\ZF$ saying that every universe satisfies $\ZF$, which is just intended to set the scope of the present treatment. (For most of the results, the authors believe that natural generalizations to weak fragments of $\ZF$ are possible.) Theorem \ref{Thm: multiverse conservativity} shows that an extension of $\CM$ interpreting the G\"odel-L\"ob modal logic is conservative over $\ZF$. So by the soundness principle, $\textnormal{\sf Multiverse}_\ZF$, and G\"odel's second incompleteness theorem, $\CM$ does not prove the statement $\exists \mathcal{U} \dt \Uni(\mathcal{U})$, saying that there exists a universe. 

A flexible revision-semantic technique for expanding models of the background theory $\ZF$ to models of extensions of $\CM^-$ is developed. This technique builds on ideas from \cite{Gup82}, \cite{Her82a} and \cite{Her82b}, for circumventing truth-theoretic paradoxes. In short, one starts by setting parameters specifying the particular multiverse conception desired. Among other things, this pins down the interpretation of $\Uni$. Then the interpretation of $\Mod$ is determined by a revision-semantic process. Intuitively, a basic definition of truth-in-a-universe is supplied among the parameters, and this definition is revised step-by-step to more adequate definitions. Theorems \ref{Thm: Con^omega interprets NEC + Non-triv} and \ref{Thm: R^omega interprets T} show that some natural settings of the parameters lead to that further semantically motivated axioms and deductive rules are validated in the constructed model, more on this further below.

We introduce several axioms and deductive rules in response to the Main Question. The most fundamental such principle for this framework is the deductive rule of Necessitation, $\NEC$, which is a formal expression of the Copernican Principle: 
\[
\text{If $\sigma$ is provable, then $\forall \mathcal{U} \dt (\Uni(\mathcal{U}) \rightarrow \Mod(\mathcal{U}, \gquote{\sigma}))$ is provable,}
\]
where $\gquote{\sigma}$ is the G\"odel code of $\sigma$. Under mild assumptions on the parameters, $\NEC$ is validated in the revision-semantic model construction. Theorem \ref{Thm: GL interpretation} shows that $\CM + \NEC$ is conservative over $\ZF$.

Dually, the deductive rule of Co-Necessitation, $\CONEC$, states: 
\[
\text{If $\forall \mathcal{U} \dt (\Uni(\mathcal{U}) \rightarrow \Mod(\mathcal{U}, \gquote{\sigma}))$ is provable, then $\sigma$ is provable.}
\]
In the context of $\NEC$ as formalizing the Copernican Principle, $\CONEC$ may be thought of as expressing that the theory is maximal within the bounds of the Copernican Principle. On the other hand, as a stand-alone principle, $\CONEC$ can be used to boost the expressive power: For example, we will consider $\CM$ extended by $\CONEC$ and the statement that no universe satisfies a $\Sigma^0_1$-statement that does not already hold in the standard model of arithmetic in the background theory; in other words, a Turing machine that does not halt in the background theory, halts in no universe. In $\CM$ extended with this axiom we can use basic model-theoretic considerations to prove that every universe satisfies the Reflection schema iterated $\omega^\textrm{CK}_1$ times over $\ZF$.\footnote{See System \ref{Sys: Reflection} for the definition of this theory.} Now, by adding $\CONEC$ we can prove $\omega^\textrm{CK}_1$-iterated Reflection schema over $\ZF$ outright in the background theory. So in general, $\CONEC$ enables outright proofs of statements that are provably satisfied across a model-theoretically delimited multiverse, thus in some sense ``extracting the deductively accessible content'' of higher-order non-recursive properties. 

We write $\MS$ (Multiverse theory of Satisfaction) for the theory $\CM + \NEC + \CONEC$. This theory is analogous to the the Firedman---Sheard theory of truth ($\FS$) from \cite{FS87}. A revision-semantic technique for constructing models of $\CM + \NEC \text{ and/or } \CONEC$ (building on a technique from the aforementioned two papers) is embodied in the Main Lemma (in \S \ref{Sec: Interpret}) and its Corollary \ref{Cor: T interprets MS}.

We now proceed to discuss three axioms motivated by the Main Question that have a reflective character in that they assert that the background universe is in some sense reflected in the multiverse. We will establish bounds on the consistency strength of these in terms of iterated reflection principles. The reader is referred to Systems \ref{Sys: GR omega} and \ref{Sys: Reflection} for the definition of these principles.

A very basic multiverse axiom is $\textnormal{\sf Non-Triviality}$, $\exists \mathcal{U} \dt \Uni(\mathcal{U})$, saying that there is a universe. In the presence of $\NEC$, this also yields that every universe contains a universe, and so on. We show in Theorem \ref{Thm: Con^omega interprets NEC + Non-triv} (and in Theorem \ref{Thm: multiverse conservativity}) that $\CM + \textnormal{\sf Non-Triviality} + \NEC$ is locally interpreted in (and conservative over) the theory of iterated consistency over $\ZF$.

A stronger axiom motivated by the Main Question, called $\textnormal{\sf Self-Perception}$, expresses that the background universe is isomorphic (over the set theoretic language) to one of the internal universes. This embodies the idea that the universe of the background theory should also be available in its internal multiverse, and has a distinct reflective character. It turns out to be convenient to take the universes to be countable recursively saturated models when modelling $\CM + \textnormal{\sf Self-Perception} + \NEC$, a phenomenon that corresponds to the multiverse model of Gitman and Hamkins in (\citeyear{GH10}). This suggests that their multiverse theory would harmonize well with $\CM + \textnormal{\sf Self-Perception} + \NEC$, a hypothesis we explore briefly in subsection \ref{Subsec: Hamkins's multiverse}. In Theorem \ref{Thm: CM + NEC + self-iso}, we use the revision-semantic technique to interpret $\CM + \textnormal{\sf Self-Perception} + \NEC$ in the theory of $\omega$-iterated Global Reflection over $\ZF$. The latter is a natural untyped theory of truth, that mildly strengthens $\ZF$. All universes are countable recursively saturated in this interpretation.

$\textnormal{\sf Self-Perception}$ is closely related to the notion of condensible models studied by Enayat in (\citeyear{En20}). The definition of condensability is somewhat technical, involving an infinitary language: A model $\mathcal{M}$ of $\ZF$ is {\em condensible}, if there is some ordinal $\alpha \in \mathcal{M}$ such that $\mathcal{M} \cong \mathcal{M}(\alpha) \prec_{\mathbb{L}_\mathcal{M}} \mathcal{M}$, where $\mathcal{M}(\alpha)$ denotes the substructure of $\mathcal{M}$ of ranks below $\alpha$ and $\mathbb{L}_\mathcal{M}$ denotes the intersection of $\lang_{\omega_1, \omega}$ with the well-founded part of $\mathcal{M}$. In particular, Enayat positively answers a question that sprung from the present paper: Is there an $\omega$-standard model of $\ZF$ with unboundedly many ordinals $\alpha$ such that $\mathcal{M} \cong \mathcal{M}(\alpha) \prec \mathcal{M}$? Note that recursively saturated models are $\omega$-non-standard. Enayat's result means that the door also appears to be open for models of $\CM + \textnormal{\sf Self-Perception} + \NEC$ with $\omega$-standard universes.

We also introduce the axiom schema of $\textnormal{\sf Multiverse Reflection}$, stating for each sentence $\sigma$ in the language $\lang$ of set theory: $\big( \forall \mathcal{U} \dt (\Uni(\mathcal{U}) \rightarrow \Mod(\mathcal{U}, \gquote{\sigma})) \big) \rightarrow \sigma$. Over $\CM$, this principle is implied by $\textnormal{\sf Self-Perception}$ and implies $\textnormal{\sf Non-Triviality}$. Using the revision-semantic technique, we show in Theorem \ref{Thm: R^omega interprets T} (and Theorem \ref{Thm: multiverse conservativity}) that $\MS + \textnormal{\sf Multiverse Reflection}$ is locally interpreted in (and conservative over) the theory of $\omega$-iterated proof-theoretic reflection schema over $\ZF$.

The body of the paper ends with case studies, where we look at two independent multiverse conceptions through the lens of the framework we have developed. The first of these is a conception of the multiverse as being arithmetically absolute, in the sense that arithmetic truth does not vary across the multiverse. The second is a conception due to Hamkins which is fundamentally based on the principles that the multiverse is closed under the forcing and inner model techniques, and that every universe is countable and $\omega$-non-standard from the perspective of some other universe. We close with a concluding section reflecting on the contributions of the paper and the relevance of this framework for future research on the set-theoretic multiverse.

\section{Preliminaries}

\subsection{A term-calculus for representation of syntax} \label{subsection: Term-calculus}

We shall work with various recursively enumerable set theories, in languages obtained by adding finitely many new non-logical symbols to the usual language of set theory on the signature $\{\in\}$. We define a {\em set theory} to be any recursively enumerable system proving Mac Lane set theory (excluding Choice)\footnote{Its axioms are Extensionality, Null set, Pair, Union, Power set, Separation for $\Delta_0$-formulas, and Infinity.} 
or proving Kripke--Platek set theory with Infinity\footnote{Its axioms are Extensionality, Null set, Pair, Union, Separation for $\Delta_0$-formulas, Collection for $\Delta_0$-formulas, Foundation for $\Pi_1$-formulas, and Infinity.}, in a language with finitely many non-logical symbols including a term-calculus for arithmetic and G\"odel coding explained below. Our specific choice of set theories underlying this definition is somewhat arbitrary (even weaker theories may well suffice). Both of these theories are sufficient for constructing the structure of the natural numbers and implementing basic model theory; these are the important features for this paper.

Since we will be reasoning about syntactic objects, it is convenient to employ a G\"odel coding of syntax. Let $K$ be a language with finitely many non-logical symbols. In any set theory $T$ (in language $L$) under consideration, we can define the arithmetic functions needed to formulate a natural {\em G\"odel coding in $T$} of terms and formulas of $K$. Through the G\"odel coding, the ``grammatical structure'' of $K$ is coherently represented in $T$. The complicated details of this procedure are described in any rigorous account of G\"odel's incompleteness theorems. The gist is that for each syntactic object (symbol, term or formula) $s$ of $K$, there is a definable number $\gquote s$ in $L$ (the G\"odel code of $s$), which represents $s$ in $T$, and there are operations definable in $T$ corresponding to syntactic operations on such objects. The authors trust that the reader is familiar with this. 

It is customary in set theory to informally introduce defined constant, relation and function symbols to the language, in order to make the presentation more readable. For example, one may use a function symbol $+$, as if it belonged to the language and there was an axiom expressing that $+$ is addition on the finite von Neumann ordinals. In this paper we assume that a finite number of such symbols needed for arithmetic and G\"odel coding are already present in the language of every set theory, and that the appropriate axioms regulating them are available in every set theory. Here follows a semi-formal account of some of the main principles of this expanded language $L$ for a set theory $T$, also serving to specify the notation:
\begin{enumerate}
\item We have a constant $\underline{0}$ and function symbols $S, +, \times$ for the successor, addition and multiplication operations in arithmetic. For each $n \in \mathbb{N}$, $\underline{n}$ is shorthand for $S^n(\underline{0})$. 
\item Each variable, constant, relation or function symbol $s$ of $K$ is represented in $T$ by a numeral $\gquote{s}$ in $L$. \item Recursively, each term $f(t_1, \dots, t_n)$ of $K$ is represented by the term $\gquote{f}(\gquote{t_1}, \dots, \gquote{t_n})$ of $L$. Formally, the term $\gquote{f}(\gquote{t_1}, \dots, \gquote{t_n})$ is the result of applying a function symbol of $L$ to the numerals $\gquote{f}, \gquote{t_1}, \dots, \gquote{t_n}$. Moreover, $\gquote{f(t_1, \dots, t_n)}$ denotes a numeral that $T$ proves to equal $\gquote{f}(\gquote{t_1}, \dots, \gquote{t_n})$.
\item Each atomic formula $R(t_1, \dots, t_n)$ of $K$ is represented by the numeral $\gquote{R(t_0, \dots, t_n)}$ of $L$. Analogous remarks apply as in the case of terms described above.
\item The syntactic operations, standardly used to build up complex formulas from atomic formulas, are all available. For example:
\begin{enumerate}
\item $L$ has a function symbol $\negd$, such that for each $\phi$ in $K$, $\negd \gquote{\phi}$ represents $\neg \phi$. Moreover, $T$ proves that $\gquote{\neg \phi}=\negd \gquote{\phi}$.
\item $L$ has a function symbol $\wedged$, such that for each $\phi$ and each $\psi$ in $K$,  $\gquote{\phi} \wedged \gquote{\psi}$ represents $\phi \wedge \psi$. Moreover, $T$ proves that $\gquote{\phi \wedge \psi}=\gquote{\phi} \wedged \gquote{\psi}$.
\item $L$ has a function symbol $\foralld$, such that for each variable $v$ and each formula $\phi$ in $K$,  $\foralld \gquote{v} \dt \gquote{\phi}$ represents $\forall v \dt \phi$. Moreover, $T$ proves that $\gquote{\forall v \dt \phi}=\foralld \gquote{v} \dt \gquote{\phi}$.
\item For any $\phi$ in $K$, $\phi[t/x]$ denotes the formula obtained from $\phi$ by replacing each free occurrence of the variable $x$ by the term $t$ (if $t$ has variables, then their bound occurrences in $\phi$ are renamed as necessary). $L$ has a function symbol (written $-[-\udot{/}-]$) which represents this primitive recursive substitution operation. Moreover, $T$ proves that $\gquote{\phi(x)[y/x]} = \gquote{\phi(x)}[\gquote{y}\udot{/}\gquote{x}]$. Somewhat less formally, if $\phi$ has been introduced as $\phi(x)$, we may write $\phi(t)$ for the formula $\phi[t/x]$.
\end{enumerate}
\end{enumerate}
In the context of a set theory $T$ in a set-theoretic language $L$, $\Sigma^0_n$, $\Pi^0_n$ and $\Delta^0_n$ denote the usual arithmetic hierarchy as defined for $L$-formulas (all quantifiers are bounded to $\mathbb{N}$), up to equivalence in $T$. It is well-known that for any recursive system $S$, there is a $\Sigma^0_1$-formula $\Pr_{\udot S}$, representing $S$-provability in $T$. We write $\Con_{\udot S}$ for the sentence $\neg\Pr_{\udot S}(\gquote\bot)$, expressing that $S$ is consistent. In both cases, the dot under $S$ is sometimes omitted, when it can be inferred from the context.

As an example, consider this consequence of G\"odel's second incompleteness theorem: 
\[
\ZF \not\vdash \Con_{\udot{\ZF}}
\] 
From the perspective of the meta-theory, ``$\ZF$'' refers to a set of sentences (the object theory of $\ZF$) whereas ``$\udot{\ZF}$'' refers to a formula representing the recursive set of G\"odel codes of that set in the object theory $\ZF$.

Suppose now that a set theory $T'$ in language $L'$ is represented in a set theory $T$ in language $L$. Then $L'$ is G\"odel coded in $T$, as explained above. But $T'$, in turn, also G\"odel codes languages; say that $T'$ G\"odel codes the language $L''$. Note that the whole G\"odel coding of $L''$ in $T'$ is then carried along by the representation of $T'$ in $T$. For example, if $\phi$ is a formula in $L''$, then there is a term $\gquote \phi$ in $L'$ which represents $\phi$ in $T'$. If $\psi(x)$ is a formula of $L'$, we can then form the formula $\psi[\gquote \phi / x]$ of $L'$. This formula, in turn, is represented in $T$ by an $L$-term $\gquote{\psi[\gquote \phi / x]}$. So if $\theta(y)$ is an $L$-formula, we can form the $L$-formula $\theta[\gquote{\psi[\gquote \phi / x]} / y]$. Thus, G\"odel codes may be nested, as a set theory represents a set theory, which in turn represents a language.

\subsection{Miscellaneous logical preliminaries}\label{subsection: Misc prel}

$\lang$ is the language with the symbol ``$\in$'' along with a finite number of arithmetic and syntactic symbols as explained in Subsection \ref{subsection: Term-calculus}. We assume that $\ZF$ is formulated as an $\lang$-theory, with the natural axioms for defining the arithmetic and syntactic symbols of $\lang$. $\lang^+$ denotes any extension of $\lang$ with a finite number of new symbols.

If $L$ is a language and $S_1, \cdots, S_n$ are symbols, then $L_{S_1, \cdots, S_n}$ denotes the language obtained by augmenting $L$ with $S_1, \cdots, S_n$. The Separation schema applying to all formulas of a language $L$ is denoted $\Sep(L)$, and the Replacement schema applying to all formulas of a language $L$ is denoted $\Rep(L)$.

Any set theory suffices as meta-theory. Suppose that in the meta-theory we consider a definable set $A = \{x \mid \phi(x)\}$, such as a theory. We may then refer to the corresponding set within an object-theory, for example as follows: Using the symbol $A$ somewhat ambiguously, we write a statement of the form $\ZF \vdash \cdots \mathcal{M} \models \udot{A} \cdots$ for the more formally precise statement of the form $\ZF \vdash \cdots \exists X \dt \big( \forall x (x \in X \leftrightarrow \phi(x)) \wedge \forall x \in X \dt (\mathcal{M} \models x) \big) \cdots$. The dot under $A$ is occasionally omitted, when clear from the context. To illustrate, we might express a special case of G\"odel's completeness theorem within $T$ as $T \vdash \big( \Con_{{\ZF}} \leftrightarrow \exists \mathcal{M} (\mathcal{M} \models {\ZF}) \big)$. We say that a theory $T_1$ bounds the {\em consistency strength} of (or has at least as high consistency strength as) a theory $T_0$ if the consistency of $T_1$ implies the consistency of $T_0$.

An {\em interpretation} $\mathcal{I}$ from a language $L_0$ to a language $L_1$ is a function $\mathcal{I} : L_0 \rightarrow L_1$ which is generated (by structural recursion) from the non-logical symbols of $L_0$. Moreover, we say that $\mathcal{I}$ {\em interprets} or {\em validates} the $L_0$-system $T_0$ in the $L_1$-system $T_1$ if for any $L_0$-formula $\phi$, $T_1 \vdash \mathcal{I}(\phi)$ whenever $T_0 \vdash \phi$.

As default, we work with first-order languages and classical logic, but we will consider additional deductive rules (NEC and CONEC). $\phi \equiv \psi$ is the statement that $\phi$ and $\psi$ are identical formulas. If $S$ and $T$ are systems in languages both including $L$, then $S \equiv_L T$ is the statement that $S$ and $T$ have the same $L$-theorems. If $S$ is a system involving deductive rules, and $A$ is an axiom, then $S + A$ denotes the natural extension of $S$ in which these deductive rules may be applied to proofs also involving $A$. For example, in $\MS + \exists x \dt \Uni(x)$, we may use $\NEC$ to derive $\forall \mathcal{U} \in \Uni\dt \Mod(\mathcal{U}, \gquote{\exists x \dt \Uni(x)})$. 

It is sometimes notationally convenient to introduce classes of the form $A = \{ x \mid \phi(x) \}$ (the class of all sets $x$ such that $\phi(x)$), where $\phi$ is an $\lang^+$-formula. Then $x \in A$ may be regarded as an alternative notation for $\phi(x)$. Thus, we have no need to specify a formal class theory. For example $V = \{ x \mid \top\}$ is the class of all sets. 

Formally, $\var$ is the set of variables $\{{x}_1,  {x}_2, \dots \}$, indexed by the positive natural numbers. But we use other symbols informally (such as $x, y, p, f, \mathcal{U}, \cdots$) for variables as well. $\VA$ is the class of variable-assignments, $\{f \mid f : \var \rightarrow V \}$. If $a$ is a set (or a structure), then $\VA^a$ is the set $\{f \mid f : \var \rightarrow a\}$ of all variable-assignments to elements of (the underlying set of) $a$. If $f$ is a variable-assignment and $v$ is a variable, then $\VA_{f, v}$ is the set of all variable-assignments $g$, such that for all $u \in \var$, $u \neq v \rightarrow g(u) = f(u)$. Suppose that we are working in a set theory $T$ in a language $L$ containing terms $t_1, \cdots t_n$. Note that for $n < \omega$, $T$ proves from $v_1 \in \var, \cdots, v_n \in \var$ that there is a primitive recursive variable-assignment $f$ satisfying $f(v_1) = t_1, \cdots, f(v_n) = t_n$, and $\forall m \in \mathbb{N} \dt (m > n \rightarrow f(v_m) = \underline{0})$. Such a variable assignment $f$ is denoted $\langle v_1, \dots v_n \rangle \mapsto \langle t_1, \cdots t_n \rangle$ (or just $v_1 \mapsto t_1$ in the case $n = 1$).

We assume that model theory is set up so that any structure $\mathcal{M}$ uniquely determines its language, which we denote by $L(\mathcal{M})$, and we take the symbol ``$\mathcal{M}$'' to refer ambiguously to both the structure and its domain. Let $\mathcal{M}$ be an $L$-structure, $\phi$ a formula in $L$ and $f \in \VA^\mathcal{M}$. We use the arrow-notation $\vec{a}$ for finite tuples $\langle a_1, \cdots, a_n \rangle$, and the shorthand $\vec{a} \in \mathcal{M}$ for that each component $a_i$ of $\vec{a}$ is an element of $\mathcal{M}$. We write $\mathcal{M} \models(\phi, f)$ for the statement ``$\phi$ is true in $\mathcal{M}$ under the variable-assignment $f$'', as defined in the usual Tarskian semantics of first-order logic. If $\vec{a} \in \mathcal{M}$ and $\psi(\vec{x}) \in L$, then we write $\mathcal{M} \models \psi(\vec{a})$ for $\mathcal{M} \models (\psi, \vec{x} \mapsto \vec{a})$. We write $\mathcal{M} \models \phi$ for $\forall f \in \VA^\mathcal{M} \dt \mathcal{M} \models (\phi, f)$. If $K$ is a sublanguage of $L$, then $\mathcal{M}\restr_K$ denotes the reduct of $\mathcal{M}$ to $K$. We write $\mathcal{M} \equiv_K \mathcal{N}$ for the statement that $\mathcal{M}$ satisfy the same $K$-sentences as $\mathcal{N}$. We write $\mathcal{M} \cong_K \mathcal{N}$ for the statement that $\mathcal{M}\restr_K$ is isomorphic to $\mathcal{N}\restr_K$. When the subscripts are dropped, they are assumed to be $L(\mathcal{M})$.

We use abbreviations for certain variations of the quantifiers: 
\begin{enumerate}
\item  $\exists x \in y \dt \phi$  stands for  $\exists x \dt (x \in y \wedge \phi)$ . 
\item  $\forall x \in y \dt \phi$  stands for  $\forall x \dt (x \in y \rightarrow \phi)$. 
\item  $\exists^! x \dt \phi$  stands for  $\exists x \dt (\phi(x) \wedge \forall y \dt (\phi(y) \rightarrow x = y))$. 
\end{enumerate}
If $P$ is a predicate symbol, we may write $x \in P$ for $P(x)$. Similarly, we write  $\exists x \in P \dt \phi$  for  $\exists x \dt (P(x) \wedge \phi)$ , and so on.

We will introduce primitive relation symbols ``$\Sat$'', ``$\Uni$'' and ``$\Mod$''. Informally, $\Sat(\phi, f)$ expresses that $\phi$ is satisfied under the variable assignment $f$; $\Uni(\mathcal{U})$ expresses that $\mathcal{U}$ is a universe; and $\Mod(\mathcal{U}, \phi, f)$ expresses that $\phi$ is satisfied in the universe $\mathcal{U}$ under the variable assignment $f$. Recall that $\mathcal{L}_\Sat$ denotes the language $\mathcal{L}$ augmented with the symbol $\Sat$, while $\lang_{\Uni, \Mod}$ denotes $\mathcal{L}$ augmented with $\Mod$ and $\Uni$.

Again, we introduce some abbreviations:
\begin{enumerate}
\item  $\Sat(\phi)$  and  $\Tr(\phi)$  stand for  $\forall f \in \VA \dt \Sat(\phi, f)$.
\item  $\Mod(\mathcal{U}, \phi)$  stands for  $\forall f \in \VA^\mathcal{U} \dt \Mod(\mathcal{U}, \phi, f)$.
\item  $\Tr^\Box(\phi)$  stands for   $\forall \mathcal{U} \in \Uni\dt \forall f \in \VA^\mathcal{U} \dt \Mod(\mathcal{U}, \phi, f)$.
\item  $\Tr^\Diamond(\phi)$  stands for   $\neg \Tr^\Box(\dot\neg \phi)$.
\end{enumerate}

If $X$ is a formula, term or definable object in the language $L$ of the structure $\mathcal{M}$, then $X^\mathcal{M}$ denotes its interpretation in $\mathcal{M}$; e.g. $\phi^\mathcal{M} =_\df \{\vec{a} \in \mathcal{M} \mid \mathcal{M} \models \phi(\vec{a})\}$. 

Informally speaking, if $\mathcal{M}$ is a model of set theory, $\mathcal{N} \in \mathcal{M}$, and $\mathcal{M} \models $ ``$\mathcal{N}$ is a model of set theory'', so that $\mathcal{N}$ is an internal model of $\mathcal{M}$, then we may need to extract $\mathcal{N}$ into an external model that can be studied on a par with $\mathcal{M}$. This is achieved by the following formal definition.

\begin{dfn}\label{Dfn: externalization}
If $\mathcal{M}$ is a model of set theory and $a \in \mathcal{M}$, then 
\[
a_\mathcal{M} =_\df \{x \in \mathcal{M} \mid \mathcal{M} \models x \in a\}.
\] 
This notation is generalized in the cases that $a$ is considered as a relation or as a structure: If $\mathcal{M}$ is a model of set theory, $R \in \mathcal{M}$ and $\mathcal{M}$ satisfies that $R$ is an $\underline{n}$-ary relation (for a natural number $n$ under consideration), then 
\begin{align*}
R_\mathcal{M} =_\df \big\{ \langle x_1, \cdots, x_n \rangle \mid \hspace{2pt} & \exists p \in \mathcal{M} \dt [ \mathcal{M} \models p \in R \hspace{2pt} \wedge \\
& \bigwedge_{1 \leq i \leq n} \text{``$x_i$ is the $\underline{i}$:th component of $p$''}]\big\}.
\end{align*}
If $\mathcal{M}$ is a model of set theory, $\mathcal{N}, N, R_1, \cdots, R_n$ are elements of $\mathcal{M}$, and $\mathcal{M}$ satisfies that $\mathcal{N}$ is a structure with domain $N$ and relations $R_1, \cdots, R_n$ (of arities $\underline{r_1}, \cdots, \underline{r_n}$, respectively), then 
\[
\mathcal{N}_\mathcal{M} =_\df \langle N_\mathcal{M}, (R_1)_\mathcal{M}, \cdots, (R_n)_\mathcal{M} \rangle.
\]
In either of the above cases $a_\mathcal{M}$ is called the {\em $\mathcal{M}$-externalization of $a$}.
\end{dfn}

For example, $\omega^\mathcal{M}$ denotes the {\em element} of $\mathcal{M}$ such that $\mathcal{M} \models \phi(\omega)$, where $\phi$ defines $\omega$. On the other hand, $(\omega^\mathcal{M})_\mathcal{M}$ denotes the {\em subset} $\{a \in \mathcal{M} \mid \mathcal{M} \models a \in \omega\}$ of $\mathcal{M}$, consisting of all $a$ such that $\mathcal{M} \models a \in \omega$.

\subsection{Recursive saturation}\label{subsection: Rec sat}

A {\em type $p(\vec{x})$, over a theory $T$} in a language $L$, is a set of $L$-formulas such that $T \cup p$ is consistent when the variables $\vec{x}$ are considered as fresh constant symbols. If $\mathcal{M} \models T$, then $p(\vec{x})$ is {\em realized} in $\mathcal{M}$ if there is $\vec{a} \in \mathcal{M}$, such that for all $\phi(\vec{x}) \in p$, we have $\mathcal{M} \models \phi(\vec{a})$. A {\em type $p(\vec{x}, \vec{b})$, over $\mathcal{M}$ with parameters $\vec{b} \in \mathcal{M}$}, is a type over $\mathrm{Th}(\mathcal{M}, \vec{b})$ (the theory of $\mathcal{M}$ with parameters $\vec{b}$). Such a type $p(\vec{x}, \vec{b})$ is {\em recursive} if it is a recursive set (under some fixed G\"odel coding of the formulas as natural numbers). A structure $\mathcal{M}$ is {\em recursively saturated} if it realizes every recursive type over $\mathcal{M}$. A {\em $\crsm$} is a countable recursively saturated model.

\begin{thm}[Completeness of the $\crsm$-semantics]\label{Thm: rec sat complete}
Let $\mathcal{M}$ be a countable model in a recursive language. There is a countable recursively saturated elementary extension of $\mathcal{M}$. In particular, every consistent theory in a recursive language is modelled by a $\crsm$.
\end{thm}
\begin{proof}
See the proof of Theorem 2.4.1 in \cite{CK90}.
\end{proof}

Let $\mathcal{M}$ be a model of a set theory $T$. The interpretations of the numerals in $\mathcal{M}$ are called the {\em standard natural numbers of $\mathcal{M}$}. For each $n < \omega$, let us make the identification $n = \underline{n}^\mathcal{M}$. We say that $\mathcal{M}$ is {\em $\omega$-non-standard} if there is $c \in (\omega^\mathcal{M})_\mathcal{M} \setminus \omega$. Such a $c$ is said to be a {\em non-standard number of $\mathcal{M}$}. 

\begin{prop}\label{Prop: rec sat omega-non-standard}
If $\mathcal{M}$ is a recursively saturated model of a set theory, then $\mathcal{M}$ is $\omega$-non-standard.
\end{prop}
\begin{proof}
By recursive saturation, $\mathcal{M}$ realizes the type $\{x \in \mathbb{N}\} \cup \{\underline{n} < x \mid n \in \mathbb{N}\}$.
\end{proof}

Suppose that $\mathcal{M}$ is $\omega$-non-standard. We say that a subset $A$ of $\omega$ is {\em coded in $\mathcal{M}$}, if there is (a {\em code}) $a \in (\omega^\mathcal{M})_\mathcal{M}$, such that $A = \{n \in \mathbb{N} \mid \mathcal{M} \models \underline{n} < a\}$. We define the {\em standard system of $\mathcal{M}$} as
\[
\SSy(\mathcal{M}) =_\df \{A \subseteq \omega \mid \text{``$A$ is coded in $\mathcal{M}$''}\}.
\]

The following result is due to Wilmers (\citeyear{Wil75}), employing Friedman's back-and-forth technique (\citeyear{Fri73}):
\begin{thm}[Canonicity of countable recursively saturated models] \label{Thm: rec sat iso}
Let $\mathcal{M}$ and $\mathcal{N}$ be $\crsm$s modelling $\ZF$. If $\mathcal{M} \equiv_\lang \mathcal{N}$ and $\SSy(\mathcal{M}) = \SSy(\mathcal{N})$, then $\mathcal{M} \cong_\lang \mathcal{N}$. 
\end{thm}
\begin{proof}
See the proof of Theorem 7.14 in \cite{KG19}.
\end{proof}

\begin{rem*} As far as the authors can see, the proof of the above Theorem requires both the full Separation and Replacement schemas of $\ZF$, as well as its Foundation axiom.
\end{rem*}

We define the class
\[
\crsmb =_\df \{\mathcal{M} \mid \mathcal{M}\restr_\lang \models \ZF \wedge \textnormal{``$\mathcal{M}$ is a $\crsm$''}\}.
\]

\subsection{Systems of satisfaction over $\ZF$}

Before embarking on developing a framework for a notion of truth-in-a-universe relevant to the set-theoretic multiverse, we shall go through some related systems of truth over $\ZF$. An intuitive philosophical perspective is that systems of truth capture various absolute notions of truth, as motivated by the universe view of set theory, while our framework for truth-in-a-universe captures various relative notions of truth, as motivated by the multiverse view of set theory. From a mathematical perspective, it is interesting to relate these two approaches. Moreover, since we will generalize techniques that have been developed for studying systems of truth, these provide a relevant context for viewing our results.

Right at the start of the endeavour to axiomatize truth, one faces the choice between introducing (to the base language of set theory) a unary truth-predicate $\Tr(\sigma)$, applying to sentences $\sigma$, or a binary satisfaction-relation $\Sat(\phi, f)$, applying to formulas $\phi$ and variable-assignments $f : \var \rightarrow V$. In the former option, $\sigma$ needs to range over a class-sized language where there is a constant-symbol $c_x$ corresponding to each $x \in V$. The authors have chosen the latter option. As a general heuristic, one is justified to expect any theory of satisfaction to be interpretable in a corresponding theory of truth; the idea being to interpret $\Sat(\phi(x), f)$ by $\Tr(\phi[c_{f(x)}/x])$. Fujimoto has made a comprehensive study of theories of truth over set theory, following the former option (\citeyear{Fuj12}). 

\begin{sys}[$\CT$]\label{Sys: CT}
Let $S$ be a set theory in $\lang^+$. The system $\CT(S)\restr$ (for Compositional Truth) consists of these axioms in the language $\lang^+_\Sat$:
\[
\begin{array}{ll}
\mathsf{Base} & S  \\
\CT_= & \forall y_0, y_1 \dt \big( \Sat(\gquote{x_0 = x_1}, \langle \gquote{x_0}, \gquote{x_1} \rangle \mapsto \langle y_0, y_1 \rangle) \leftrightarrow y_0 = y_1 \big) \\
\CT_\in & \forall y_0, y_1 \dt \big( \Sat(\gquote{x_0 \in x_1}, \langle \gquote{x_0}, \gquote{x_1} \rangle \mapsto \langle y_0, y_1 \rangle) \leftrightarrow y_0 \in y_1 \big) \\
\CT_\neg & \forall \phi \in \lang^+_\Sat \dt \forall f \in \VA \dt (\Sat(\udot\neg \phi, f) \leftrightarrow \neg \Sat(\phi, f)) \\
\CT_\wedge & \forall \phi, \psi \in \lang^+_\Sat \dt \forall f \in \VA \dt (\Sat(\phi \wedged \psi, f) \leftrightarrow \Sat(\phi, f) \wedge \Sat(\psi, f)) \\
\CT_\forall & \forall \phi \in \lang^+_\Sat \dt \forall f \in \VA \dt (\Sat(\udot\forall u \dt \phi, f) \leftrightarrow \forall g \in \VA_{f, u} \dt \Sat(\phi, g)) \\
\end{array}
\]
We write $\CT\restr$ for $\CT(\ZF)\restr$. The axioms of the form $\CT_-$ are called {\em compositional axioms}. By basic logic, $\CT\restr$ also proves the axioms $\CT_\vee$, $\CT^\rightarrow$ and $\CT_\exists$ (analogously defined). We use phrases such as ``$\Sat$ is $\vee$-compositional'' to express that we have $\CT_\vee$, for example.

$\CT$ is $\CT\restrp \Sep(\lang_\Sat) + \Rep(\lang_\Sat)$. 
\end{sys}

A routine induction argument in the meta-theory shows this Proposition: 
\begin{prop}\label{Prop: T biconditionals}
For all $\lang$-formulas $\phi(x)$, $\CT\restr \hspace{3pt} \vdash \Sat(\gquote{\phi(x)}, \gquote{x} \mapsto y) \leftrightarrow \phi(y)$.
\end{prop}

The theory of satisfaction $\CT\restrp \Sep(\lang_\Sat)$ corresponds to the theory of truth $\mathsf{TC}\restrp \Sep^+$ in \cite[\S 4]{Fuj12}. It is straightforward to interpret the former in the latter. Using this, it follows from Theorem 20 in \cite[\S 4.1]{Fuj12} that $\CT\restrp \Sep(\lang_\Sat)$ is conservative over $\ZF$. In contrast, in $\CT$ we have access to the Reflection theorem for $\lang_\Sat$-formulas, enabling us to prove that there is a $V_\alpha$ modelling $\ZF$. See \cite[\S 4.1]{Fuj12} for more details and refinements.

\begin{prop}\label{Prop: CT induction}
$\CT\restrp \Sep(\lang_\Sat) \vdash \textnormal{``Transfinite Induction over $\lang_\Sat$''}$
\end{prop}
\begin{proof}
Let $\phi(x) \in \lang_\Sat$. Assuming $\neg \phi(\alpha)$, for some ordinal $\alpha$, we shall refute the corresponding induction hypothesis. By $\Sep(\lang_\Sat)$, the set $S = \{\xi \leq \alpha \mid \neg\phi(\xi)\}$ exists, and by assumption it is non-empty. Let $\beta$ be the least ordinal in $S$. Then $\neg \phi(\beta)$ and $\forall \xi < \beta \dt \phi(\xi)$, refuting the induction hypothesis.
\end{proof}

\begin{sys}[$\GR^\omega$]\label{Sys: GR omega}
Let $S$ be a set theory in $\lang^+$. Here we present the systems of iterated Global Reflection over $S$, denoted $\GR^\alpha(S)$, for ordinals $\alpha \leq \omega$. For any set theory $T$ in the language $\lang_\Sat$, the axiom of {\em Global Reflection} over $T$ is 
\[
\begin{array}{ll}
\GR_{\udot T} & \forall \phi \in \udot\lang^+_\Sat \dt (\Prv_{\udot T}(\phi) \rightarrow \Tr(\phi)).
\end{array}
\]
(The dot under $T$ is sometimes omitted, when it can be inferred from the context.)

Recursively, for each $\alpha \leq \omega$, we define the system $\GR^\alpha(S)$: 
\begin{align*}
\GR^0(S) &=_\df \CT(S) \restrp \Sep(\lang^+_\Sat) \\
\GR^{\alpha+1}(S) &=_\df \GR^\alpha(S) + \GR_{\GR^{\alpha}} \\
\GR^\omega(S) &=_\df \bigcup_{n < \omega} \GR^n(S)
\end{align*}
We write $\GR^\alpha$ for $\GR^\alpha(\ZF)$.
\end{sys}

\begin{rem*} Observe that $\GR^\alpha$ is defined with $\CT \restrp \Sep(\lang_\Sat)$ as base case, thus without Replacement for formulas with the Satisfaction predicate. The reason for this is that it is intended to express iterated Global Reflection over $\ZF$, which $\CT \restrp \Sep(\lang_\Sat)$ conservatively extends. If it were defined with $\CT$ as base case, then it would not be morally ``over $\ZF$'', since $\CT$ proves strong reflection principles of its own.
\end{rem*}

\begin{sys}[$\FS$]
The systems $\FS\hspace{-3pt}\restriction$ and $\FS$ (for Friedman--Sheard) are obtained by adding these rules of proof to $\CT\hspace{-3pt}\restriction$ and $\CT$, respectively:
\[
\begin{array}{ll}
\NEC & \text{For each $\phi \in \sent(\lang_\Sat)$: If $\FS \vdash \phi$, then $\FS \vdash \Tr(\gquote{\phi})$.} \\
\CONEC & \text{For each $\phi \in \sent(\lang_\Sat)$: If $\FS \vdash \Tr(\gquote{\phi})$, then $\FS \vdash \phi$.} \\
\end{array}
\]
Recall that $\Tr(\theta)$ is defined as $\forall f \in \VA \dt \Sat(\theta, f)$.
\end{sys}
Given a set-theoretic system $S$ in language $L$, the following rule will be considered:
\[
\begin{array}{ll}
\textnormal{\sf Reflection rule} & \text{For each $\phi$ in $L$: If $S \vdash \Pr_S(\gquote{\phi})$, then $S \vdash \phi$.}
\end{array}
\]

\begin{dfn}
A set-theoretic system $S$ in language $L$ is $\omega$-inconsistent if there is an $L$-formula $\uranus(x)$ such that:
\begin{align*}
S &\vdash \forall x \dt (\uranus(x) \rightarrow x \in \mathbb{N}) \\
S &\vdash \exists x \dt \uranus(x) \\
\text{For each } n \in \mathbb{N} \text{, } S &\vdash \neg\uranus(\underline{n})
\end{align*} 
We say that $\uranus(x)$ {\em witnesses the $\omega$-inconsistency of $S$}.\footnote{The symbol ``$\uranus$'' (Uranus)  used here is meant to help the reader remember that its extension consists of non-standard numbers: The arrow may be taken to point outwards from the standard model $\mathbb{N}$.}
\end{dfn}

\begin{prop}\label{Prop: omega-consistent reflection rule}
If $\ZF$ is $\omega$-consistent, then $\ZF$ is closed under the {\sf Reflection rule}.
\end{prop}
\begin{proof}
Suppose $\ZF \not\vdash \phi$. Then $\ZF \vdash \neg \Pr_\ZF(\underline{n}, \gquote{\phi})$, for every standard $n \in \mathbb{N}$ (where $\Pr_\ZF(\underline{n}, \gquote{\phi})$ is the formula expressing that $\underline{n}$ is the G\"odel code of a proof of $\phi$). Now, by $\omega$-consistency, $\ZF \vdash \neg \Pr_\ZF(\gquote{\phi})$.
\end{proof}

Let $c$ be a fresh constant symbol. Note that the schema $\{\underline{m} < c < \omega \mid m \in \mathbb{N}\}$, expressing that there is a non-standard natural number, yields $\omega$-inconsistency when added to a set-theoretic system (proof: take $x = c$ as $\uranus$).

\begin{prop}
$\GR^\omega$ and $\FS$ are $\omega$-inconsistent. $\GR^\omega + \GR_{\GR^\omega}$ and $\FS + \GR_{\FS}$ are inconsistent.
\end{prop}
\begin{proof}
This is a corollary of McGee's paradox, see \cite{McG85}, and can be proved analogously as Theorem 13.9 and Corollary 14.39 in \cite{Hal14}, respectively. The arguments in \cite{Hal14} are written for theories of truth over $\PA$, not for theories of satisfaction over $\ZF$. But they go through with these natural modifications:
\begin{enumerate}
\item Replacing instances of the truth predicate ``$T$'' by our defined predicate ``$\Tr$'', except for instances quantifying over terms, of the form $\forall \vec{t} \dt T(\phi \udot[ \vec{t} \udot/ \gquote{\vec{x}} \udot])$, which are replaced by $\forall \vec{y} \in \omega \dt \Sat(\phi, \gquote{\vec{x}} \mapsto \vec{y})$, where $\vec{y}$ is fresh.
\item Replacing quantifiers of the form ``$\forall x$'' by ``$\forall x \in \omega$''.
\end{enumerate}

These arguments rely on that $\GR^\omega$ admits $\NEC$, which we proceed to show: Let $\sigma \in \sent(\lang_\Sat)$ and suppose that $\GR^\omega \vdash \sigma$. Then there is $k < \omega$, such that $\GR^k \vdash \sigma$. since this proof can be represented in $\GR^\omega$, we have $\GR^\omega \vdash \Pr_{\GR^k}(\gquote{\sigma})$, so that by Global Reflection, $\GR^\omega \vdash \Tr(\gquote{\sigma})$. Since $\GR^\omega \vdash \CT\restr$, it follows that $\GR^\omega \vdash \sigma$, as desired.
\end{proof}

So if we were to naturally extend the definition of $\GR^\alpha$ to all ordinals $\alpha$, then we would get that $\GR^\alpha$ is inconsistent for all $\alpha > \omega$.

Later on we will introduce a multiverse axiom, called $\textnormal{\sf Self-Perception}$, to the effect that the universe of the background theory is isomorphic (over $\lang$) to one of its internal universes; this axiom is motivated by the idea that the universe of the background theory should be available in its multiverse. The following lemmas establish technical results needed to validate that axiom.

\begin{lemma}\label{Lemma: Internal model iso}
Let $\mathcal{U}$ be a $\crsm$ of $\GR_0$ and let $\mathcal{V} \in \mathcal{U}$, such that $\mathcal{U}$ satisfies 
\[
\mathcal{V} \models \{\sigma \in \sent(\lang_\Sat) \mid \Tr(\sigma)\}.
\]
Then $\mathcal{U} \cong_\lang \mathcal{V}_\mathcal{U}$.
\end{lemma}
Recall that $\mathcal{V}_\mathcal{U}$ is the $\mathcal{U}$-externalization of $\mathcal{V}$, see Definition \ref{Dfn: externalization}.
\begin{proof}
We shall establish $\mathcal{U} \cong_\lang \mathcal{V}_\mathcal{U}$ by invoking Theorem \ref{Thm: rec sat iso}. Thus we need to show that $\mathcal{V}_\mathcal{U}$ is a $\crsm$, that $\SSy(\mathcal{U}) = \SSy(\mathcal{V}_\mathcal{U})$ and that $\mathcal{U} \equiv_\lang \mathcal{V}_\mathcal{U}$.

Note that $\mathcal{U}$ is $\omega$-non-standard, by $\mathcal{U} \in \crsmb$ and Proposition \ref{Prop: rec sat omega-non-standard}. That $\mathcal{V}_\mathcal{U}$ is a $\crsm$ now follows from Lemma 2.2 in \cite{GH10}.\footnote{That Lemma is stated for $\ZFC$, but it is easily seen that its proof does not make use of Choice.}

$(\omega^\mathcal{U})_\mathcal{U}$ is mapped initially into $(\omega^{(\mathcal{V}_\mathcal{U})})_{(\mathcal{V}_\mathcal{U})}$ by an embedding $j$ (for each $x \in (\omega^\mathcal{U})_\mathcal{U}$, $j(x)$ is defined as the unique $y \in \mathcal{U}$ such that $\mathcal{U} \models y = \underline{x}^\mathcal{V}$). Therefore, we obtain $\SSy(\mathcal{U}) = \SSy(\mathcal{V}_\mathcal{U})$ as follows: Let $A \in \SSy(\mathcal{U})$, coded by $a \in (\omega^\mathcal{U})_\mathcal{U}$. Since $j$ is an embedding, $j(a)$ is a code for $A$ in $(\omega^{(\mathcal{V}_\mathcal{U})})_{(\mathcal{V}_\mathcal{U})}$. Conversely, let $B \in \SSy(\mathcal{V}_\mathcal{U})$, coded by $b \in (\omega^{(\mathcal{V}_\mathcal{U})})_{(\mathcal{V}_\mathcal{U})}$. Since $j$ is initial and $\mathcal{U}$ is $\omega$-non-standard, there is a non-standard $c \in (\omega^\mathcal{U})_\mathcal{U}$, such that $j(c) \leq^{\mathcal{V}_\mathcal{U}} b$. So since $j$ is an embedding, $c$ is a code for $B$ in $(\omega^\mathcal{U})_\mathcal{U}$. 

To see $\mathcal{U} \equiv_\lang \mathcal{V}_\mathcal{U}$, let $\phi$ be a sentence of $\mathcal{L}$. By absoluteness of $\models$ for standard formulas, 
\[\mathcal{V}_\mathcal{U} \models \phi \iff \mathcal{U} \models \text{``} \mathcal{V} \models (\ulcorner \phi \urcorner) \text{''}.\]
Since $\Sat$ satisfies the Tarski-biconditionals for $\lang$, we have
\[\mathcal{U} \models \phi \iff \mathcal{U} \models \Tr(\ulcorner \phi \urcorner).\]
Moreover, by $\CT_\neg$ and the condition of the Lemma,
\[\mathcal{U} \models \Tr(\ulcorner \phi \urcorner) \iff \mathcal{U} \models \text{``} \mathcal{V} \models (\ulcorner \phi \urcorner) \text{''}.\]
By combining these we obtain $\mathcal{U} \equiv_\lang \mathcal{V}_\mathcal{U}$, as desired.
\end{proof}

\begin{rem*} $\mathcal{U} \models \ZF$ is needed for the proof of this lemma, as it relies on Theorem \ref{Thm: rec sat iso}. Thus, the authors do not expect it to generalize to $\GR_0(S)$, unless $S \vdash \ZF$.
\end{rem*}

\begin{lemma}\label{Lemma: GR internal model}
Let $k < \alpha \leq \omega$, and let $\mathcal{U} \models \GR^\alpha$. Then there is $\mathcal{V} \in \mathcal{U}$, such that $\mathcal{U}$ satisfies
\[\mathcal{V} \in \crsmb \wedge \mathcal{V} \models \{\sigma \in \sent(\lang_\Sat) \mid \Tr(\sigma)\}.\]
In particular, $\mathcal{U}$ satisfies that $\mathcal{V} \models \GR^k$. Moreover, if $\mathcal{U} \in \crsmb$, then $\mathcal{U} \cong_\lang \mathcal{V}_\mathcal{U}$.
\end{lemma}
\begin{proof}
We work in $\mathcal{U}$. From $\ZF + \Sep(\lang_\Sat)$ we get that the set $\mathbf{Tr} = \{\phi \in \sent(\lang_\Sat) \mid \Tr(\phi)\}$ exists. For the first statement, by completeness of the $\crsm$-semantics (Theorem \ref{Thm: rec sat complete}), it suffices to establish $\Con(\mathbf{Tr})$; and for this it suffices to establish $\Con(\sigma)$, where $\sigma$ is an arbitrary finite conjunction of sentences in $\mathbf{Tr}$. By $\wedge$-compositionality of $\Sat$, we have $\Sat(\sigma)$. By $\GR_{\ZF + \Sep(\lang_\Sat)}$, we have $\Prv_{\ZF + \Sep(\lang_\Sat)}(\udot\neg\sigma) \rightarrow \Sat(\udot\neg\sigma)$, and by $\neg$-compositionality of $\Sat$, we have $\Sat(\udot\neg\sigma) \leftrightarrow \neg \Sat(\sigma)$. So since $\Sat(\sigma)$, we obtain $\Prv_{\ZF + \Sep(\lang_\Sat)}(\udot\neg\sigma) \rightarrow \bot$, whence $\Con(\sigma)$. By Theorem \ref{Thm: rec sat complete}, we can let $\mathcal{V}$ be a model of $\mathbf{Tr}$ in $\crsmb$.

It follows from $\GR_{\GR^k}$ that $\GR^k \subseteq \mathbf{Tr}$, yielding the second statement.

The last statement now follows from Lemma \ref{Lemma: Internal model iso}.
\end{proof}

\begin{ax}
Let $\iota$ be a function symbol and $\self$ be a constant symbol. $\Iso(x)$ denotes an $\lang^+_\iota$-formula expressing that $x$ is an $\lang^+$-structure and that $\iota$ is an $\in$-isomorphism from the universe $V$ onto $x$. We shall study this axiom in $\lang^+_{\iota, \self}$:
\[
\Iso(\self)
\]
(This formulation is chosen over $\exists x \dt \Iso(x)$, as it is convenient to have a reference to a witness.)
\end{ax}

By the $\in$-isomorphism property, and the absoluteness of $\models$ for standard formulas, we have:
\begin{prop}\label{Prop: iso elementarily equiv}
For each $\phi(\vec{x}) \in \lang$, $\ZF + \Iso(\self) \vdash (\self \models \gquote{\phi(\iota(\vec{x}))}) \leftrightarrow \phi(\vec{x})$.
\end{prop}

\begin{sys}\label{Sys: SR}
Let $S$ be a set theory in $\lang^+$. By recursion, for each $\alpha \leq \omega$, we define the system $\SP^\alpha(S)$ (standing for Self-Perception):
\begin{align*}
\SP^0(S) &=_\df \GR^0(S) = \CT(S) \restrp \Sep(\lang^+_\Sat) \\
\SP^{\alpha+1}(S) &=_\df \GR^{\alpha+1}(S) + \self \in \crsmb + \Iso(\self) + \self \models \mathbf{Tr} \cup \udot{\SP}^{\underline{\alpha}}(\udot S) \\
\SP^\omega(S) &=_\df \bigcup_{n < \omega} \SP^n(S),
\end{align*}
where $\mathbf{Tr} =_\df \{\sigma \in \sent(\lang_\Sat) \mid \Tr(\sigma)\}$.
We write $\SP^\alpha$ for $\SP^\alpha(\ZF)$.
\end{sys}

\begin{rem*} A clarificatory note on the role of the languages in $\SP^\alpha$. Let $\alpha \geq 1$. The language of $\SP^\alpha$ is $\lang_{\Sat, \iota, \self}$. $\SP^\alpha$ proves Separation for $\lang_{\Sat, \self}$, and it proves Replacement for $\lang_\self$ (since $\self$ can be treated as a parameter in these schemas). 
\end{rem*}

\begin{lemma}\label{Lem: model of CT expands to Iso}
Let $\alpha \leq \omega$. Every $\crsm$ $\mathcal{U}$ of $\GR^\alpha$ expands to a model $\mathcal{U}^*$ of $\SP^\alpha$. Moreover, if $\mathcal{U}$ is a definable model, then we can also obtain that $\mathcal{U}^*$ is definable.
\end{lemma}
\begin{proof}
We start by showing the case $\alpha < \omega$ by induction. The base case $\alpha = 0$ is trivial. The induction hypothesis is that if $\mathcal{U} \models \GR^{\alpha}$, then $\mathcal{U}$ expands to a model of $\SP^\alpha$; and if $\mathcal{U}$ is definable, then the expansion is definable. Let $\mathcal{U} \models \GR^{\alpha + 1}$. Applying Lemma \ref{Lemma: GR internal model}, we find a $\crsm$ $\mathcal{V}$ in $\mathcal{U}$, such that $\mathcal{U} \cong_\lang \mathcal{V}_\mathcal{U}$ (as witnessed by an $\lang$-isomorphism $i : \mathcal{U} \rightarrow \mathcal{V}_\mathcal{U}$) and $\mathcal{U}$ satisfies that $\mathcal{V} \models \mathbf{Tr} \cup \udot{\GR}^{\underline{\alpha}}$. So by the induction hypothesis applied in $\mathcal{U}$, $\mathcal{U}$ satisfies that $\mathcal{V}$ expands to a model $\mathcal{W}$ of $\udot{\SP}^{\underline{\alpha}}$. Let $\mathcal{U}^*$ be the model obtained from $\mathcal{U}$ by interpreting $\self$ by $\mathcal{W}$ and interpreting $\iota$ by $i$. It is immediate from the construction that $\mathcal{U}^* \models \SP^{\alpha + 1}$.

Now to the case that $\alpha = \omega$: Assume that $\mathcal{U} \models \GR^\omega$. Working in $\mathcal{U}$, let $\mathbf{Tr} = \{\phi \in \sent(\lang_\Sat) \mid \Tr(\phi)\}$. By the above, 
\[
\{x \in \crsmb \wedge x \models \mathbf{Tr} \cup \udot{\SP}^{\underline{n}} \mid n < \omega\}
\]
is a recursive type over $\mathcal{U}$. So since $\mathcal{U}$ is recursively saturated, it is realized by some $\mathcal{W}$ in $\mathcal{U}$. It now follows from Lemma \ref{Lemma: Internal model iso} that $\mathcal{U} \cong_\lang \mathcal{W}_\mathcal{U}$. So just as in the former case, $\mathcal{U}$ can be expanded to a model $\mathcal{U}^*$ of $\SP^\omega$.

Assume now that $\mathcal{U}$ is definable. $\mathcal{W}$ can then be defined as the least element of a definable enumeration of $\mathcal{U}$ that satisfies the appropriate conditions. An isomorphism witnessing $\mathcal{U} \cong_\lang \mathcal{W}_\mathcal{U}$ can also be defined: As seen from a close look at the proof of Theorem \ref{Thm: rec sat iso}, the isomorphism is constructed by recursion on enumerations of $\mathcal{U}$ and $\mathcal{W}_\mathcal{U}$, both of which can be chosen definable since $\mathcal{U}$ and $\mathcal{W}_\mathcal{U}$ are definable and countable.
\end{proof}

\begin{lemma}\label{Lemma: GR^omega interprets SP^omega}
$\GR^\omega$ interprets $\SP^\omega$.
\end{lemma}
\begin{proof}
Since $\GR^\omega$ is $\omega$-inconsistent, there is a formula $\uranus(x)$ such that $\GR^\omega \vdash \exists x < \omega \dt\uranus(x)$, but for each $n \in \mathbb{N}$, $\GR^\omega \vdash \neg\uranus(\underline{n})$.

We start by working in $\GR^\omega$. By $\ZF + \Sep(\lang_\Sat)$, the theory $\mathbf{Tr} = \{\sigma \in \sent(\lang_\Sat) \mid \Tr(\sigma)\}$ of truth is a set; and by the argument starting the proof of Lemma \ref{Lemma: GR internal model}, it is a consistent theory. So there is a definable $\mathcal{S} \in \crsmb$ that models truth.\footnote{This follows from the proof of Theorem 2.4.1 in \cite{CK90}. The key observation to see that the model is definable is that it is essentially a Henkin-construction by recursion on an enumeration of a recursive language and on an enumeration of the set of all recursive subsets of that language, both of which can be chosen definable. } Since $\mathcal{S}$ is a model of truth, the Global Reflection axioms allow us to prove $\mathcal{S} \models \GR^{\underline{n}}$, for each standard natural number $n$. Let $d < \omega$ be the minimal number such that $\uranus(d)$, and let $c$ be the maximal number such that $c \leq d$ and $\mathcal{S} \models \GR^c$. Note that for each standard natural number $n$, we can prove $\underline{n} < c$. By Lemma \ref{Lem: model of CT expands to Iso}, $\mathcal{S}$ expands to a definable model $\mathcal{S}'$ of $\SP^c$. 

Working in the meta-theory, it follows that $\GR^\omega$ interprets $\SP^\omega$ by an interpretation $\mathcal{J}$ mapping each sentence $\sigma$ in the language of $\SP^\omega$ to the $\lang_\Sat$-sentence $\mathcal{S}'  \models \gquote{\sigma}$.
\end{proof}

\section{Revision-semantic truth-in-a-universe}\label{sec: construction}

We shall now go through the key revision-semantic technique introduced in this paper, which may be used to construct a variety of untyped truth-in-a-universe relations for the multiverse of set theory. Revision-semantics was independently invented in \cite{Gup82}, \cite{Her82a} and \cite{Her82b}. In \cite{FS87}, the axiomatic theory of truth $\FS$ was presented and shown to be validated by a model constructed through such a revision process. The revision process starts with an arbitrary extension $S_0$ of truth, and recursively defines $S_{n + 1}$ as the theory of the structure $(\mathbb{N}, S_n)$. In particular, the theory of $\mathbb{N}$ is a subset of $S_1$, and the liar sentence is in $S_{n+1}$ iff it is not in $S_{n}$. 

The construction in this paper is somewhat different in that it is intensional. We start with a more-or-less arbitrary formula {\em defining} truth-in-a-universe, and revise the definition in a revision-semantic fashion. The construction can be modified by adjusting parameters. For example, we shall see that certain conditions on the parameters result in that the eventual definition of truth-in-a-universe validates the multiverse theory $\MS$, introduced in System \ref{Sys: MS}. This theory is analogous to $\FS$, but is actually weaker. The parameters need to satisfy some basic conditions as specified in this definition:

\begin{dfn}\label{dfn: appropriate}
Let $\Trm_n(\phi)$, $\Unirm_n(\mathcal{U})$ and $\Modrm_0(\mathcal{U}, \phi, f)$ be formulas of the meta-language ($\lang$), in the free variables $\{n, \phi\}$, $\{n, \mathcal{U}\}$ and $\{\mathcal{U}, \phi, f\}$, respectively.\footnote{Even though the $n$ is notationally in subscript-position, it is a free variable of the formulas $\Trm$ and $\Unirm$. This pattern will also be used for the formula $\Modrm$ introduced below.} For each $n \in \mathbb{N}$: 
\begin{align*}
\Tb_n &=_\df \{\phi \mid \Trm_n(\phi)\} \\
\Tb_\omega &=_\df \bigcup_{n < \omega} \Tb_n \\
\Unib_n &=_\df \{\mathcal{U} \mid \Unirm_n(\mathcal{U})\} \\
\Unib_\omega &=_\df \bigcap_{n < \omega} \Unib_n
\end{align*}
$(\Tb_n)_{n \in \mathbb{N}}$ is intended to be a sequence of first-order set theories, and $(\Unib_n)_{n \in \mathbb{N}}$ is intended to be a sequence of classes of models, as formally specified below. Let $\lang^\Trm, \lang^\Rev$ be recursive languages. We say that $\Trm, \Unirm, \Modrm_0$ are {\em revision parameters (in $\lang^\Trm, \lang^\Rev$)} if the following conditions (closed under $\forall n \in \mathbb{N}$, where appropriate) are provable in the meta-theory and in a given set theory (as object theory):
\begin{enumerate}
\item $\lang \subseteq \lang^\Rev \subseteq \lang^\Trm$
\item ``The symbols $\Uni, \Mod$ do not appear in $\lang^\Trm$.''
\item ``$\Tb_0$ is a set theory in $\lang^\Trm$.''
\item $ \Tb_{n+1} \vdash \Tb_{n} $
\item $ \Unirm_0(\mathcal{U}) \rightarrow \text{``$\mathcal{U}$ is an $\lang^\Trm$-structure.''}  $
\item $ \Unib_{n+1} \subseteq \Unib_n $
\item $ \Modrm_0(\mathcal{U}, \phi, f) \rightarrow \Unirm_0(\mathcal{U}) \wedge \phi \in \lang^\Rev_{\Uni, \Mod} \wedge f \in \VA^{\mathcal{U}} $
\end{enumerate}
\end{dfn}

In the construction below we shall see how, given revision parameters, an untyped revision-semantic truth-in-a-universe predicate can be defined as an $\lang$-formula $\Modrm_n(\mathcal{U}, \phi, f)$, with this intended reading of the variables: $n$ is the stage in the revision process, $\mathcal{U}$ is a universe, $\phi$ is a formula in (the representation of) $\lang^\Rev_{\Uni, \Mod}$ and $f$ is an assignment of variables to elements of $\mathcal{U}$. So $\lang^\Trm$ is the language of the theories $\Tb_n$, $\lang^\Rev$ is any sublanguage of $\lang^\Trm$, and $\lang^\Rev_{\Uni, \Mod}$ is the language undergoing revision. The $\lang$-formula $\Modrm^\circlearrowright$ is also introduced as a variant of $\Modrm_n$. Actually, only the $\Unirm_n$ and $\Modrm_0$ parameters influence the construction. The $\Trm_n$ parameter comes into play later on, in the Main Lemma (in \S \ref{Sec: Interpret}), where we show (under certain conditions on the revision parameters) that the $\Modrm_n$ formula satisfies desirable semantically motivated axioms when constructed in the theory $\Tb_\omega$. 

The construction may intuitively be thought of as a recursive procedure, where $\Modrm_0$ is a more-or-less arbitrary truth-in-a-universe relation and each $\Modrm_{n+1}$ revises $\Modrm_n$ into a more adequate relation. It turns out to be efficient to perform the construction using G\"odel's fixed-point lemma. It is in fact possible to choose $\Modrm_0$ such that it gets revised to itself (see the definition of $\Modrm^\circlearrowright$ below). This phenomenon contrasts with the revision-semantics ordinarily used to construct a model of the Friedman--Sheard theory of truth, where the revision-operation has no fixed-point (see Lemma 14.9(iii) in \cite{Hal14}). A key difference between the present revision-process and that one is that the former is intensional and the latter is extensional. In the present framework we start with an arbitrary formula {\em defining} truth-in-a-universe and revise it to more adequate definitions, while in the other framework one starts with an arbitrary {\em extension} of truth and revise it more adequate extensions. The move from extensional to intensional revision-semantics is highly relevant for the present framework.

\begin{constr}[Construction of Revision-semantics for the Multiverse]\label{Constr: Multiverse semantics}
Let $\Trm_n(\phi)$, $\Unirm_n(\mathcal{U})$ and $\Modrm_0(\mathcal{U}, \phi, f)$ be revision parameters.

By G\"odel's fixed-point lemma, there is an $\lang$-formula $\Modrm_n(\mathcal{U}, \phi, f)$, in the free variables $n,\mathcal{U}, \phi, f$, such that provably:
\[
\Modrm_n(\mathcal{U}, \phi, f) \leftrightarrow \left(
	\begin{aligned}
	& n \in \mathbb{N} \wedge \Unirm_n(\mathcal{U}) \wedge \phi \in \lang^\Rev_{\Uni, \Mod} \wedge f \in \VA^{\mathcal{U}} \\
	\wedge & \big(n = 0 \rightarrow \Modrm_0(\mathcal{U}, \phi, f) \big) \\
	\wedge & \big( n>0 \rightarrow \langle \mathcal{U}\restr_{\lang^\Rev}, \gquote{\Unirm_{\underline{n-1}}}^\mathcal{U}, \gquote{\Modrm_{\underline{n-1}}}^\mathcal{U} \rangle \models (\phi, f) \big)
	\end{aligned}
	\right) \tag{$\dagger$}
\]
Recall that if $\phi(\vec{x})$ is a formula in the language of a structure $\mathcal{M}$, then $\phi^\mathcal{M} = \{\vec{a} \in \mathcal{M} \mid \mathcal{M} \models \phi(\vec{a})\}$. Above, this notation is used for formulas in a represented language, hence the G\"odel quotes, $\gquote{}$. Working in $\ZF$, $\langle \mathcal{U}\restr_{\lang^\Rev}, \gquote{\Unirm_{\underline{n-1}}}^\mathcal{U}, \gquote{\Modrm_{\underline{n-1}}}^\mathcal{U} \rangle$ is the expansion of $\mathcal{U}\restr_{\lang^\Rev}$ to $\lang^\Rev_{\Uni, \Mod}$, interpreting $\gquote{\Mod}$ by $\{\vec{a} \in \mathcal{U} \mid \mathcal{U} \models \gquote{\Modrm_{\underline{n-1}}}(\vec{a})\}$, and interpreting $\gquote{\Uni}$ by $\{u \in \mathcal{U} \mid \mathcal{U} \models \gquote{\Unirm_{\underline{n-1}}}(u)\}$.

Formally, we now have two references for the expression ``$\Modrm_0$'', the formula $\Modrm_0$ and the formula $\Modrm_n$, with the variable assignment $n \mapsto 0$. However, it is clear that these are equivalent.

The above construction works for a very wide range of choices for $\Modrm_0$. But by the fixed-point lemma, we can choose $\Modrm_0$ to be ``equivalent to its own revision'', so that $\Modrm_n$ turns out to be constant with respect to $n$. Indeed, there is an $\lang$-formula $\Modrm^\circlearrowright$, such that provably:
\[
\Modrm^\circlearrowright(\mathcal{U}, \phi, f) \leftrightarrow \left(
	\begin{aligned}
	& \Unirm_0(\mathcal{U}) \wedge \phi \in \lang^\Rev_{\Uni, \Mod} \wedge f \in \VA^{\mathcal{U}} \\
	\wedge & \langle \mathcal{U}\restr_{\lang^\Rev}, \gquote{\Unirm_{\underline{0}}}^\mathcal{U}, \gquote{\Modrm^\circlearrowright}^\mathcal{U} \rangle \models (\phi, f) 
	\end{aligned}
	\right)
\]
\begin{flushright} 
End of Construction.
\end{flushright}
\end{constr}

For the reader familiar with the extensional revision-procedure used to construct a model of $\FS$ (this construction is reasonably well-known for $\FS$ formulated over $\PA$, see \cite[ch. 14.1]{Hal14} for a detailed treatment), note how the recursive call in the fixed-point formula of our construction operates on an intension (the formula $\Modrm$), which is interpreted in internal models. In contrast, the recursive call of the revision procedure for constructing a model of $\FS$ (say as a theory of truth over arithmetic) operates on an extension (the set of true sentences from the previous step), which is obtained from the external model. The authors take this to explain why it is possible to define a truth-in-a-universe relation $\Modrm^\circlearrowright$ which is fixed by the revision-procedure. In the extensional revision-semantics, this is not possible simply because the liar sentence must switch truth-value in the external model at every step of the revision. 

Necessity was the mother of the intensional revision-semantics of this paper; the authors do not see any way to construct models of ``the Copernican multiverse of sets'' (as formalized by various theories in this paper, e.g. $\CM + \textsf{Non-Triviality} + \NEC$) by the extensional approach. Conversely, the authors do not see that the intensional approach can replace the extensional approach, as the former relies on that the intension acted upon in the recursive call is interpreted in an internal model. For arithmetic this may be a serious obstacle, as arithmetic does not have that kind of internal models.

Some conditions and rules for revision parameters, relevant for showing that the $\Modrm_n$ formula constructed as above satisfies a desirable semantically motivated theory (see the Main Lemma in \S \ref{Sec: Interpret}), are shown in Figure \ref{Fig: Revision}. If one of the rules holds, we say that the revision parameters admit it. Essentially, if the revision parameters admit $\NEC^*$ or $\CONEC^*$, then the multiverse theory interpreted admits $\NEC$ or $\CONEC$, respectively. In practice it is often easier to work with the other rule and conditions in Figure \ref{Fig: Revision}, using this Lemma:

\begin{figure}
\caption{Rules and conditions for the revision parameters}
\label{Fig: Revision}

\vspace{6pt}
\begin{center} 
{\bf Revision rules}
\end{center}
\[
\begin{array}{ll}
\NEC^* & \forall n \in \mathbb{N} \dt \forall \phi \in \lang^\Rev \dt \big( (\Tb_{n} \vdash \phi) \rightarrow ( \Tb_{n+1} \vdash \forall \mathcal{U} \in \Unib_{n+1} \dt (\mathcal{U} \models \gquote{\phi})) \big) \\
\CONEC^* & \forall n \in \mathbb{N} \dt \forall \phi \in \lang^\Rev \dt \big( ( \Tb_{n+1} \vdash \forall \mathcal{U} \in \Unib_{n+1} \dt (\mathcal{U} \models \gquote{\phi})) \rightarrow \Tb_{n+2} \vdash \phi \big) \\
\textnormal{\sf Reflection rule}^* & \forall n \in \mathbb{N} \dt \forall \phi \in \lang^\Rev \dt \big(( \Tb_{n+1} \vdash \Pr_{\Tb_{n}}(\gquote{\phi})) \rightarrow \Tb_{n+2} \vdash \phi \big)
\end{array}
\]

\begin{center} 
{\bf Revision conditions}
\end{center}
\[
\begin{array}{ll}
\textnormal{\sf Soundness}^* & \forall n \in \mathbb{N} \dt \forall \mathcal{U} \in \Unib_{n+1} \dt (\mathcal{U} \models \Tb_{n}) \\
\textnormal{\sf Completeness}^* & \forall n \in \mathbb{N} \dt \forall \phi \in \lang^\Rev \dt \big( ( \forall \mathcal{U} \in \Unib_{n+1} \dt (\mathcal{U} \models \phi)) \rightarrow \Tb_{n}\vdash \phi \big) 
\end{array}
\]

\end{figure}

\begin{lemma}\label{Lemma: Admissible rules on revision parameters}
Let $\Trm, \Unirm, \Modrm_0$ be revision parameters. 
\begin{enumerate}[{\rm (a)}]
\item If $\textnormal{\sf Soundness}^*$ is provable, then the parameters admit $\NEC^*$.
\item If $\textnormal{\sf Completeness}^*$ is provable and the parameters admit the $\textnormal{\sf Reflection rule}^*$, then the parameters admit $\CONEC^*$.
\end{enumerate}
\end{lemma}
\begin{proof}
\begin{enumerate}[{\rm (a)}]
\item Let $n \in \mathbb{N}$ and $\phi \in \lang^\Rev$. Assume that $\textnormal{\sf Soundness}^*$ is provable (in the meta-theory $\ZF$) and that $\Tb_n \vdash \phi$. By the latter, encoding the proof in $\ZF$, we have $\ZF \vdash \Pr_{\Tb_n}(\gquote{\phi})$. Combining these, we get $\ZF \vdash \forall \mathcal{U} \in \Unib_{n+1} \dt (\mathcal{U} \models \gquote{\phi}))$. Since $\Tb_{n+1} \vdash \ZF$, we are done.

\item Let $n \in \mathbb{N}$ and $\phi \in \lang^\Rev$. Assume that $\textnormal{\sf Completeness}^*$ is provable and the parameters admit the $\textnormal{\sf Reflection rule}^*$. Since $\Tb_{n+1} \vdash \ZF$, we have $\Tb_{n+1} \vdash \textnormal{\sf Completeness}^*$. Now suppose that $\Tb_{n+1} \vdash \forall \mathcal{U} \in \Unib_{n+1} \dt (\mathcal{U} \models \gquote{\phi})$. Then $\Tb_{n+1} \vdash \Pr_{\Tb_n}(\gquote{\phi})$. So by the $\textnormal{\sf Reflection rule}^*$, $\Tb_{n+2} \vdash \phi$, as desired. \qedhere
\end{enumerate}
\end{proof}

\section{Theories of untyped satisfaction for the multiverse}\label{Sec: multiverse theory}

Section \ref{sec: construction} showed how a revision-semantic relation of truth-in-a-universe can be constructed in set theory. We turn now to the task of finding appropriate axioms for truth-in-a-universe that are validated by such revision-constructions.

\begin{sys}[$\CM$]\label{Sys: CM}
$\CM^-$, standing for {\em Compositional satisfaction for the Multiverse}, is axiomatized as follows:
\[
\begin{array}{ll}
\mathsf{Base} & \ZF + \Sep(\lang_{\Uni, \Mod})  + \Rep(\lang_{\Uni, \Mod})  \\
\CM_= & \forall \mathcal{U} \in \Uni\dt \forall f \in \VA^\mathcal{U} \dt  \big(\Mod(\mathcal{U}, \gquote{x = y}, f) \leftrightarrow f(x) = f(y))\big) \\
\CM_\neg & \forall \mathcal{U} \in \Uni\dt \forall \phi \in \lang_{\Uni, \Mod} \dt \forall f \in \VA^\mathcal{U} \dt  \big(\Mod(\mathcal{U}, \udot\neg\phi, f) \leftrightarrow \neg \Mod(\mathcal{U}, \phi, f) \big) \\
\CM_\wedge & \forall \mathcal{U} \in \Uni\dt \forall \phi, \psi \in \lang_{\Uni, \Mod} \dt \forall f \in \VA^\mathcal{U} \dt  \big(\Mod(\mathcal{U}, \phi \udot\wedge \psi, f) \leftrightarrow (\Mod(\mathcal{U}, \phi, f) \wedge \Mod(\mathcal{U}, \psi, f)) \big) \\
\CM_\forall & \forall \mathcal{U} \in \Uni\dt \forall \phi \in \lang_{\Uni, \Mod} \dt \forall f \in \VA^\mathcal{U} \dt  \big(\Mod(\mathcal{U}, \udot\forall u \dt \phi, f) \leftrightarrow \forall g \in \VA^\mathcal{U}_{f, u} \dt \Mod(\mathcal{U}, \phi, g) \big) \\
\end{array}
\]

Define $\ZF_{\Uni, \Mod} =_\df \ZF + \Sep(\lang_{\Uni, \Mod})  + \Rep(\lang_{\Uni, \Mod})$. We write $\CM$ for $\CM^-$ plus the axiom:
\[
\begin{array}{ll}
\mathsf{Multiverse}_\ZF & \forall \mathcal{U} \in \Uni \dt \forall \sigma \in \udot{\ZF}_{\Uni, \Mod} \dt \Mod(\mathcal{U}, \sigma) \\
\end{array} 
\]

If $\lang'$ expands $\lang$, then we write $\CM^-(\lang')$ and $\CM(\lang')$ for the corresponding systems obtained by replacing all occurrences of $\lang_{\Uni, \Mod}$ in the axioms of the form $\CM_-$ above by the language $\lang'_{\Uni, \Mod}$. (So the Separation and Replacement schemas remain unchanged, ranging only over $\lang_{\Uni, \Mod}$.)
\end{sys}

\begin{rem*} The natural analogue axioms $\CM_\vee, \CM_\rightarrow, \CM_\exists$ are easily derived in $\CM^-$.
\end{rem*}

\begin{rem*} In $\CM^-$, each $\mathcal{U} \in \Uni$ may be viewed as an $\lang_{\Uni, \Mod}$-structure, by performing this assignment: 
\begin{align*}
\in^\mathcal{U} &=_\df \big\{\langle a, b \rangle \mid \Mod(\mathcal{U}, \gquote{x \in y}, \langle x, y \rangle \mapsto \langle a, b \rangle)\big\} \\
\Uni^\mathcal{U} &=_\df \big\{a \mid \Mod(\mathcal{U}, \gquote{\Uni(x)}, x \mapsto a)\big\} \\
\Mod^\mathcal{U} &=_\df \big\{\langle a, b, c \rangle \mid \Mod(\mathcal{U}, \gquote{\Mod(x, y, z)}, \langle x, y, z \rangle \mapsto \langle a, b, c \rangle)\big\}
\end{align*}
\end{rem*}
Accordingly, we will occasionally use the notation $\mathcal{U} \models \phi$ for satisfaction in that $\lang_{\Uni, \Mod}$-structure. Using the compositional axioms of $\CM^-$, it is easily shown that $\Mod(\mathcal{U}, \phi, f) \iff \mathcal{U} \models (\phi, f)$.

In applications, it is natural to add further axioms to $\CM^-$, ensuring e.g. that we can prove:
\[
\begin{array}{ll}
\textnormal{\sf Non-Triviality} & \exists \mathcal{U} \Uni(\mathcal{U}) \\
\end{array}
\]
Note that over $\CM^-$, {\sf Non-Triviality} is equivalent to $\Tr^\Box(\gquote{\bot}) \rightarrow \bot$. Recall that the formulas $\Tr^\Box(\sigma)$ and $\Tr^\Diamond(\sigma)$ are defined as $\forall \mathcal{U} \in \Uni\dt (\Mod(\mathcal{U}, \sigma))$ and $\neg \Tr^\Box(\udot\neg \sigma)$, respectively. We may naturally consider the interpretation of the modal operators $\Box, \Diamond$, generated by interpreting $\Box \sigma$ by $\Tr^\Box(\gquote{\sigma})$. Therefore, it is useful to exhibit some compositional conditions easily provable for $\Tr^\Box$ in $\CM^-$ and $\CM^- + \textnormal{\sf Non-trivility}$:

\begin{prop}\label{Prop: CM box comp}
$\CM^-$ proves:
\[
\begin{array}{ll}
  \CM^\Box_\rightarrow & \forall \phi, \psi \in \sent(\lang_{\Uni, \Mod}) \dt \big( \Tr^\Box(\phi \udot\rightarrow \psi) \rightarrow (\Tr^\Box(\phi) \rightarrow \Tr^\Box(\psi)) \big) \\
  \CM^\Box_\leftrightarrow & \forall \phi, \psi \in \sent(\lang_{\Uni, \Mod}) \dt \big( \Tr^\Box(\phi \udot\leftrightarrow \psi) \rightarrow (\Tr^\Box(\phi) \leftrightarrow \Tr^\Box(\psi)) \big) \\
  \CM^\Box_\wedge & \forall \phi, \psi \in \sent(\lang_{\Uni, \Mod}) \dt \big( \Tr^\Box(\phi \udot\wedge \psi) \leftrightarrow (\Tr^\Box(\phi) \wedge \Tr^\Box(\psi)) \big) \\
    \Diamond_\CM & \forall \phi \in \sent(\lang_{\Uni, \Mod}) \dt \big( \Tr^\Diamond(\phi) \leftrightarrow \exists \mathcal{U} \in \Uni\dt \Mod(\mathcal{U}, \phi) \big) \\
\end{array}
\]
\end{prop}

\begin{prop}\label{Prop: CM + Non-triv box comp}
$\CM^- + \textnormal{\sf Non-Triviality}$ proves:
\[
\begin{array}{ll}
    \CM^\Box_\bot & \Tr^\Box(\gquote{\bot}) \leftrightarrow \bot \\
    \CM^\Box_\neg & \forall \phi \in \sent(\lang_{\Uni, \Mod}) \dt \big( \Tr^\Box(\udot\neg \phi) \rightarrow \neg \Tr^\Box(\phi) \big) \\
    \mathsf{D}_\CM & \forall \phi \in \sent(\lang_{\Uni, \Mod}) \dt \big( \Tr^\Box(\phi) \rightarrow \Tr^\Diamond(\phi) \big) \\
\end{array}
\]
\end{prop}

\begin{lemma}[Soundness Lemma]\label{Lemma: CM soundness}
$\CM^-$ proves that for all $\mathcal{U} \in \Uni$, $\{ \phi \in \lang_{\Uni, \Mod} \mid \Mod(\mathcal{U}, \phi) \}$ is deductively closed.
\end{lemma}
\begin{proof}
Using the compositional axioms $\CM_\neg$, $\CM_\wedge$ and $\CM_\forall$, this is proved just like the soundness theorem for the usual semantics of first-order logic.
\end{proof}

The theory $\MS$ (standing for Multiverse theory of Satisfaction) is analogous to the Friedman--Sheard theory of truth $\FS$:

\begin{sys}[$\MS$]\label{Sys: MS}
Consider these rules of proof:
\[
\begin{array}{ll}
\NEC & \text{For each $\phi \in \sent(\lang_{\Uni, \Mod})$: If $\MS \vdash \phi$, then $\MS \vdash \Tr^\Box(\gquote{\phi})$.} \\
\CONEC & \text{For each $\phi \in \sent(\lang_{\Uni, \Mod})$: If $\MS \vdash \Tr^\Box(\gquote{\phi})$, then $\MS \vdash \phi$.} \\
\end{array}
\]
The system $\MS^-$ is $\CM^- + \NEC + \CONEC$ and the system $\MS$ is $\CM + \NEC + \CONEC$.

If $\lang'$ expands $\lang$, then we write $\MS^-(\lang')$ and $\MS(\lang')$ for the corresponding systems obtained by replacing all occurrences of $\lang_{\Uni, \Mod}$, in the axioms of the form $\CM_-$ and in the rules $\NEC, \CONEC$, by the language $\lang'_{\Uni, \Mod}$.
\end{sys}

Recall that if $S$ is a system involving deductive rules, and $A$ is an axiom, then $S + A$ denotes the natural extension of $S$ in which these deductive rules may be applied to proofs also involving $A$. For example, in $\MS + \exists x \dt \Uni(x)$ we may use $\NEC$ to derive $\forall \mathcal{U} \in \Uni \dt \Mod(\mathcal{U}, \gquote{\exists x \dt \Uni(x)})$.

Figure \ref{Fig: Semantically motivated multiverse axioms} displays reflective axioms, modal axioms and the system $\GL_{\CM^-}$ interpreting G\"odel-L\"ob logic. 

The reflective axioms may be viewed as statements, of increasing strength, that the universe of the background theory is reflected in the multiverse: {\sf Non-Triviality} just says that there is a universe in the multiverse; $\textnormal{\sf Multiverse Reflection}$ is equivalent to that any $\mathcal{L}$-sentence holding in the background universe also holds in some universe; and $\textnormal{\sf Self-Perception}$ goes as far as saying that the background universe is isomorphic to a universe in the multiverse.

\begin{figure}
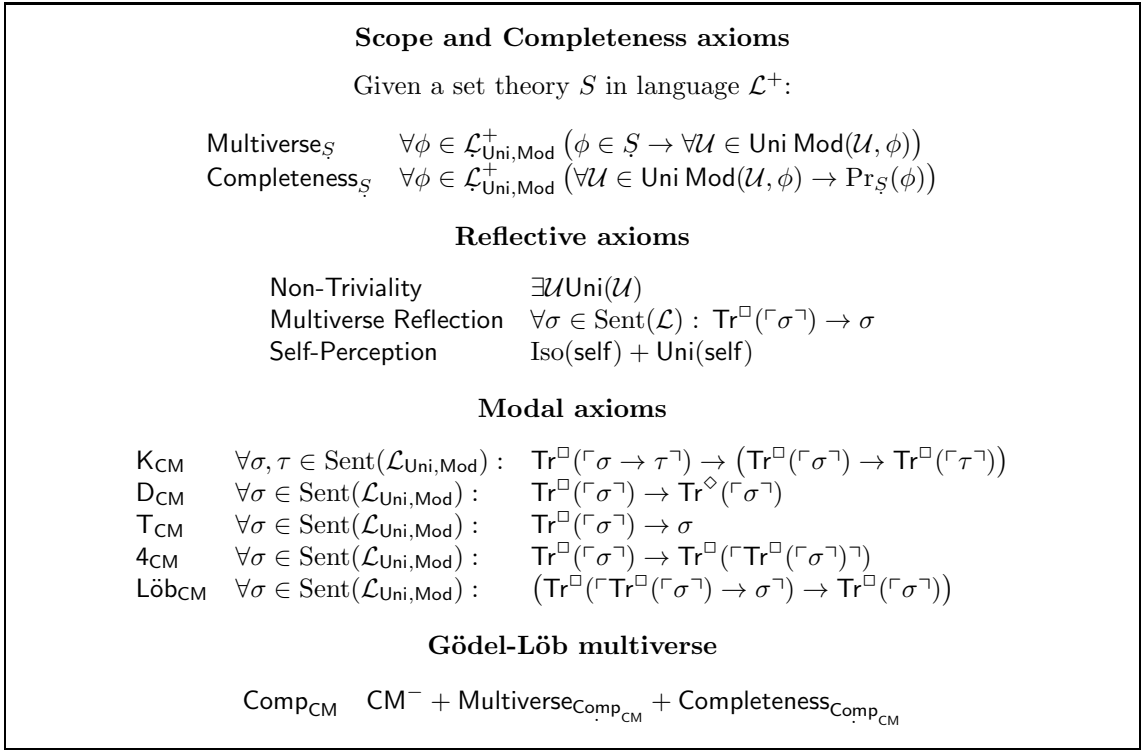

\caption{Semantically motivated multiverse axioms}
\label{Fig: Semantically motivated multiverse axioms}

\vspace{6pt}
\begin{center} 
{\bf Scope and Completeness axioms} \\
\vspace{6pt}
Given a set theory $S$ in language $\lang^+$:
\end{center}
\[
\begin{array}{ll}
\mathsf{Multiverse}_{\udot{S}} & \forall \phi \in \udot\lang^+_{\Uni, \Mod} \dt \big(\phi \in \udot S \rightarrow \forall \mathcal{U} \in \Uni \dt \Mod(\mathcal{U}, \phi) \big) \\
\mathsf{Completeness}_{\udot{S}} & \forall \phi \in \udot\lang^+_{\Uni, \Mod} \dt \big(\forall \mathcal{U} \in \Uni \dt \Mod(\mathcal{U}, \phi) \rightarrow \Pr_{\udot S}(\phi)  \big)
\end{array}
\]

\begin{center} 
{\bf Reflective axioms}
\end{center}
\[
\begin{array}{ll}
\textnormal{\sf Non-Triviality} & \exists \mathcal{U} \Uni(\mathcal{U}) \\
\textnormal{\sf Multiverse Reflection} & \forall \sigma \in \sent(\lang): \dt \Tr^\Box(\gquote{\sigma}) \rightarrow \sigma  \\
\textnormal{\sf Self-Perception} & \mathrm{Iso}(\self) + \Uni(\self)
\end{array}
\]

\begin{center} 
{\bf Modal axioms}
\end{center}
\[
\begin{array}{lll}
\textnormal{\sf K}_\CM & \forall \sigma, \tau \in \sent(\lang_{\Uni, \Mod}): & \Tr^\Box(\gquote{\sigma \rightarrow \tau}) \rightarrow \big( \Tr^\Box(\gquote{\sigma}) \rightarrow \Tr^\Box(\gquote{\tau}) \big)  \\
\textnormal{\sf D}_\CM & \forall \sigma \in \sent(\lang_{\Uni, \Mod}): & \Tr^\Box(\gquote{\sigma}) \rightarrow \Tr^\Diamond(\gquote{\sigma}) \\
\textnormal{\sf T}_\CM & \forall \sigma \in \sent(\lang_{\Uni, \Mod}): & \Tr^\Box(\gquote{\sigma}) \rightarrow \sigma  \\
\textnormal{\sf 4}_\CM & \forall \sigma \in \sent(\lang_{\Uni, \Mod}): & \Tr^\Box(\gquote{\sigma}) \rightarrow \Tr^\Box(\gquote{\Tr^\Box(\gquote{\sigma})}) \\
\textnormal{\sf L\"ob}_\CM & \forall \sigma \in \sent(\lang_{\Uni, \Mod}): & \big( \Tr^\Box(\gquote{\Tr^\Box(\gquote{\sigma}) \rightarrow \sigma }) \rightarrow \Tr^\Box(\gquote{\sigma}) \big) \\
\end{array}
\]

\begin{center} 
{\bf G\"odel-L\"ob multiverse}
\end{center}
\[
\begin{array}{lll}
\mathsf{Comp}_\CM & \CM^- + \mathsf{Multiverse}_{\udot{\mathsf{Comp}}_\CM} + \mathsf{Completeness}_{\udot{\mathsf{Comp}}_\CM} \\
\end{array}
\]
\end{figure}

The next theorem applies standard arguments from axiomatic theories of truth to exhibit semantically motivated axioms that turn out to be paradoxical.

\begin{thm}\label{Thm: liar}
\begin{enumerate}[{\rm(a)}]
\item The following axiom schema is inconsistent over $\CM^- + \NEC$:
\[
\begin{array}{ll}
\mathsf{T}_\CM & \forall \sigma \in \sent(\lang_{\Uni, \Mod}), \dt \Tr^\Box(\gquote{\sigma}) \rightarrow \sigma
\end{array}
\]
\item The following axiom schema is inconsistent over $\CM^- + \NEC + \textnormal{\sf Non-Triviality}$:
\[
\begin{array}{ll}
\mathsf{4}_\CM & \forall \sigma \in \sent(\lang_{\Uni, \Mod}), \dt \Tr^\Box(\gquote{\sigma}) \rightarrow \Tr^\Box(\gquote{\Tr^\Box(\gquote{\sigma})}) \\
\end{array}
\]
\end{enumerate}
\end{thm}
\begin{rem*} Note that $\mathsf{T}_\CM$ is the untyped version of {\sf Multiverse Reflection}.
\end{rem*}
\begin{proof}
By G\"odel diagonalization, there is an $\lang_{\Uni, \Mod}$-sentence $\lambda$, such that 
\[\CM \vdash \lambda \leftrightarrow \neg \Tr^\Box(\gquote{\lambda}).\]

By $\mathsf{T}_\CM$, $\Tr^\Box(\gquote{\lambda}) \rightarrow \lambda$, so we get $\neg \Tr^\Box(\gquote{\lambda})$, and therefore $\lambda$. Now $\Tr^\Box(\gquote{\lambda})$ follows by $\NEC$, a contradiction.

By $\mathsf{4}_\CM$, $\Tr^\Box(\gquote{\lambda}) \rightarrow \Tr^\Box(\gquote{\Tr^\Box(\gquote{\lambda})})$, so by $\CM^\Box_\leftrightarrow$, $\Tr^\Box(\gquote{\lambda}) \rightarrow \Tr^\Box(\gquote{\neg\lambda})$, and by $\CM^\Box_\neg$ (using {\sf Non-Triviality}), $\Tr^\Box(\gquote{\lambda}) \rightarrow \neg \Tr^\Box(\gquote{\lambda})$. So we get $\neg \Tr^\Box(\gquote{\lambda})$ and therefore $\lambda$. Now $\Tr^\Box(\gquote{\lambda})$ follows by $\NEC$, a contradiction.
\end{proof}

The following Proposition relates the natural systems obtained by adding reflective axioms to ${\CM^-} + \NEC$.

\begin{prop}\label{Prop: Basic multiverse axiom relationships}
Over ${\CM^-} + \NEC$:
\begin{enumerate}[{\rm(a)}]
\item $\textnormal{\sf Multiverse Reflection} \vdash \textnormal{\sf Non-Triviality}$
\item $\textnormal{\sf Self-Perception} \vdash \textnormal{\sf Multiverse Reflection}$
\end{enumerate}
\end{prop}
\begin{proof}
\begin{enumerate}[{\rm(a)}]
\item From $\Tr^\Box(\gquote{\bot}) \rightarrow \bot$ we get $\neg \forall \mathcal{U} \in \Uni\dt \Mod(\mathcal{U}, \gquote{\bot})$, whence {\sf Non-Triviality}.

\item Let $\sigma \in \lang$, and assume $\Tr^\Box(\gquote{\sigma})$. Then $\Mod(\self, \gquote{\sigma})$. So by Proposition \ref{Prop: iso elementarily equiv}, we have $\sigma$.
\end{enumerate}
\end{proof}

$\mathsf{Comp}_\CM$ in Figure \ref{Fig: Semantically motivated multiverse axioms} includes two axioms that refer to the theory $\mathsf{Comp}_\CM$. These axioms can be constructed by G\"odel's fixed-point lemma. The following theorem shows that $(\Box \mapsto \Tr^\Box)$ generates an  interpretation of G\"odel-L\"ob provability logic in $\mathsf{Comp}_\CM$, which in turn is interpretable in $\ZF$.

\begin{thm}\label{Thm: GL interpretation}
\begin{enumerate}[{\rm(a)}]
\item There is an interpretation $\mathcal{B}$ of the system $\mathsf{K}$ of modal propositional logic\footnote{The system {\sf K} is regulated by the rule of modal necessitation, $\vdash \sigma \Rightarrow \hspace{3pt} \vdash \Box \sigma$, and the axiom schema {\sf K}, $\Box (\sigma \rightarrow \theta) \rightarrow (\Box \sigma \rightarrow \Box \theta)$.} in $\CM^- + \NEC$, satisfying any given assignment of the propositional variables, and
\begin{align*}
&\mathcal{B}(\Box \sigma) = \Tr^\Box(\udot{\mathcal{B}}(\gquote{\sigma})) \text{, for each modal propositional formula $\sigma$.}
\end{align*}
\item $\mathcal{B}$ above interprets the system $\mathsf{GL}$ of modal predicate logic\footnote{The system {\sf GL} extends {\sf K} with the axiom schema {\sf 4}, $\Box \sigma \rightarrow \Box \Box \sigma$, and with {\sf L\"ob's schema}, $\Box (\Box \sigma \rightarrow \sigma) \rightarrow \Box \sigma$.} in $\mathsf{Comp}_\CM$. In particular, 
\[
\mathsf{Comp}_\CM \vdash \mathsf{K}_\CM + \mathsf{4}_\CM + \textnormal{\sf L\"ob}_\CM + \NEC.
\]
\item\label{Item: ZF interprets Comp} $\ZF$ interprets $\mathsf{Comp}_\CM$. If $\ZF$ is closed under the {\sf Reflection rule} (or if $\ZF$ is $\omega$-consistent), then $\ZF$ interprets $\mathsf{Comp}_\CM + \CONEC$.
\end{enumerate}
\end{thm}
\begin{proof}
\begin{enumerate}[{\rm(a)}]
\item The interpretation $\mathcal{B}$ can be constructed by primitive recursion, using the technique described in \cite{Hal14}[Ch. 5.3]. By $\NEC$, $\mathcal{B}$ validates the modal necessitation rule, and by $\CM^\Box_\rightarrow$ (from which $\mathsf{K}_\CM$ is easily derived), it validates {\sf K}.
\item By the Soundness Lemma, 
\[
\mathsf{Comp}_\CM \vdash \forall \sigma \in \sent(\lang_{\Uni, \Mod}) \dt (\Pr{}_{\GL_{\CM}}(\sigma) \leftrightarrow \Tr^\Box(\sigma)). \tag{*}
\] 
We apply the Hilbert-Bernays-L\"ob provability conditions. For $\NEC$, note that for each $\sigma \in \lang_{\Uni, \Mod}$, $\mathsf{Comp}_\CM \vdash \sigma \Rightarrow \mathsf{Comp}_\CM \vdash \Pr_{\mathsf{Comp}_\CM}(\gquote{\sigma})$, and apply (*). For $\mathsf{4}_\CM$,  note that for each $\sigma \in \lang_{\Uni, \Mod}$, $\mathsf{Comp}_\CM \vdash \Pr_{\mathsf{Comp}_\CM}(\gquote{\sigma}) \rightarrow \Pr_{\mathsf{Comp}_\CM}(\gquote{\Pr_{\mathsf{Comp}_\CM}(\gquote{\sigma})})$, and apply (*) both externally and internally. By L\"ob's Theorem (see Lemma 13.7 in \cite{Hal14}) and the preceding item, we are done. 
\item Let $\mathcal{C}$ be the interpretation generated by:
\begin{align*}
\Uni(\mathcal{U}) &\mapsto \mathcal{U} \models \mathsf{Comp}_\CM \\
\Mod(\Uni, \phi, f) &\mapsto \mathcal{U} \models (\phi, f)
\end{align*}
By the Tarskian conditions of satisfaction, $\ZF \vdash \mathcal{C}(\CM^-)$. By construction of $\mathcal{C}$, $\ZF \vdash \mathcal{C}(\mathsf{Multiverse}_{\mathsf{Comp}_\CM}).$ By construction of $\mathcal{C}$ and the Completeness theorem, $\ZF \vdash \mathcal{C}(\mathsf{Completeness}_{\mathsf{Comp}_\CM}).$ From (*), it is easily seen that the {\sf Reflection rule} yields $\CONEC$. By Proposition \ref{Prop: omega-consistent reflection rule}, $\omega$-consistency suffices.
\end{enumerate}
\end{proof}

\section{Interpreting the Copernican multiverse of sets}\label{Sec: Interpret}

Now we proceed to lay forth a technique for validating the theories in \S \ref{Sec: multiverse theory} by means of the revision-semantic construction in \S \ref{sec: construction}. The Main Lemma establishes that a variety of Copernican multiverse theories can by interpreted in a suitable hierarchy of theories. 

First we need a lemma establishing a normal form for derivation in $\MS$. The analogous result for the case of the Friedman--Sheard theory of truth (over arithmetic) was established by \cite{Bro21}. The authors are grateful to Broberg for allowing the inclusion of his proof re-worked for the system $\MS$. We write $\MS^-_\mathrm{NC} + S$ for the system whose theorems are the conclusions of Hilbert-style proofs in $\MS^- + S$ such that all applications of $\NEC$ are before all applications of $\CONEC$. 

\begin{lemma}\label{Lemma: Antons lemma}
Let $S$ be a theory. If $\MS^- + S \vdash \psi$, then there exists $\chi$ such that $\CM^- + \NEC + S \vdash \chi$ and $\CM^- + \CONEC + S + \chi \vdash \psi$.
\end{lemma}
\begin{proof}
For simplicity, we assume that all formulas in all theories and proofs considered are sentences; there is an adequate Hilbert-style proof system meeting this assumption. It suffices to show that $\MS^-_\mathrm{NC} + S \vdash \psi$. 

Let $\rho$ be a Hilbert-style proof, with $\psi_0, \cdots, \psi_{l-1}$ as rows, of $\psi_{l-1}$ in $\MS^- + S$. By induction, we may assume that $\MS^-_\mathrm{NC} + S \vdash \psi_r$, for each $r < l-1$. There are four cases to consider as to how the last row of $\rho$ is obtained:

\begin{description}
\item[(Axiom)] $\psi_{l-1}$ is an axiom.
\item[(First-order)] $\psi_{l-1}$ is derived by a rule of inference of first-order logic.
\item[(CONEC)] $\psi_{l-1}$ is derived from $\psi_{r}$ by $\CONEC$, for some $r < l-1$.
\item[(NEC)] $\psi_{l-1}$ is derived from $\psi_{r}$ by $\NEC$, for some $r < l-1$.
\end{description}

We proceed to establish $\MS^-_\mathrm{NC} + S \vdash \psi_{l-1}$ for each case.

\begin{description}
\item[(Axiom)] In this case $\psi_{l-1}$ is also an axiom of $\MS^-_\mathrm{NC} + S$ (a proof of length $1$).
\item[(First-order)] In this case $\psi_{l-1}$ is also derived by the same rule in $\MS^-_\mathrm{NC} + S$, utilizing the induction hypothesis.
\item[(CONEC)] In this case $\psi_{l-1}$ is also derived by $\CONEC$ in $\MS^-_\mathrm{NC} + S$, utilizing the induction hypothesis and that this application of $\CONEC$ is right at the end, after all applications of $\NEC$.
\item[(NEC)] This is the case requiring work. We have that $\psi_{l-1}$ is the sentence $\Tr^\Box(\gquote{\psi_r})$ and that $\MS^-_\mathrm{NC} + S \vdash \psi_r$. Let $\theta_0, \cdots, \theta_{k-1} = \psi_r$ be the rows of a Hilbert-style proof $\pi$ witnessing this. We proceed to show $\MS^-_\mathrm{NC} + S \vdash \Tr^\Box(\gquote{\theta_q})$, for each $q \leq k-1$. By induction, we may assume that this holds for each $q < k-1$. Again there are four cases to consider as to how the last row of $\pi$ is obtained:

	\begin{description}
	\item[(Axiom')] $\theta_{k-1}$ is an axiom.
	\item[(First-order')] $\theta_{k-1}$ is derived by a rule of inference of first-order logic.
	\item[(CONEC')] $\theta_{k-1}$ is derived from $\theta_{q}$ by $\CONEC$, for some $q < k-1$.
	\item[(NEC')] $\theta_{k-1}$ is derived from $\theta_{q}$ by $\NEC$, for some $q < k-1$.
	\end{description}
	
	We proceed to establish $\MS^-_\mathrm{NC} + S \vdash \Tr^\Box(\gquote{\theta_{k-1}})$ for each case.
	
	\begin{description}
	\item[(Axiom')] In this case we apply $\NEC$ to $\theta_{k-1}$ to obtain a proof of $\Tr^\Box(\gquote{\theta_{k-1}})$ (of length $2$).
	\item[(First-order')] We have $q_0 < \cdots < q_n < k-1$, such that $\{\theta_{q_0}, \cdots, \theta_{q_n}\} \vdash \theta_{k-1}$. By the induction hypothesis, $\MS^-_\mathrm{NC} + S \vdash \Tr^\Box(\gquote{\theta_{q_i}})$, for each $0 \leq i \leq n$. Clearly, these proofs can be merged (respecting the requirement on the order of the applications of $\NEC$ and $\CONEC$) into one Hilbert-style proof in $\MS^-_\mathrm{NC} + S$ in which $\Tr^\Box(\gquote{\theta_{q_0}}), \cdots, \Tr^\Box(\gquote{\theta_{q_n}})$ are derived. 
	
	Now note that by the Soundness Lemma, 
	\[\CM^- + \{\Tr^\Box(\gquote{\theta_{q_0}}), \cdots, \Tr^\Box(\gquote{\theta_{q_n}})\} \vdash \Tr^\Box(\gquote{\theta_{k-1}}).\] 
	We add a proof of that (which has no applications of $\NEC$ or $\CONEC$) to the end of the previous proof, obtaining $\MS^-_\mathrm{NC} + S \vdash \Tr^\Box(\gquote{\theta_{k-1}})$, as desired.
	\item[(CONEC')] In this case $\Tr^\Box(\gquote{\theta_{k-1}})$ equals $\theta_q$, which we already have a proof of.
	\item[(NEC')] In this case the last step of $\pi$ is obtained by $\NEC$, so there cannot be any application of $\CONEC$ in $\pi$. Hence, we can make an extra application of $\NEC$ at the end of $\pi$ to obtain a proof of $\Tr^\Box(\gquote{\theta_{k-1}})$ in $\MS^-_\mathrm{NC} + S$. \qedhere
	\end{description}	
\end{description}
\end{proof}

Before embarking on proving the Main Lemma of the paper, we introduce notation for the kind of interpretations involved. The Main Lemma encapsulates the revision-semantic construction of a model of ``the Copernican multiverse of sets''.

Let $\Trm, \Unirm, \Modrm_0$ be revision parameters. Let $S$ be a set theory in a language $L$ and let $t$ be an $L$-term such that $S \vdash t \in \mathbb{N}$. Then $\mathcal{I}^{\Unirm, \Modrm_0}_{L, t}$ denotes the interpretation of the language  $L_{\Uni, \Mod}$ into $L$ generated by interpreting $\Mod$ as $\Modrm_t$ (obtained from Construction \ref{Constr: Multiverse semantics}) and $\Uni$ as $\Unirm_t$. (Although the full notation is $\mathcal{I}^{\Unirm, \Modrm_0}_{L, t}$, with $L, \Uni, \Mod$ fixed, for example, we denote this interpretation by $\mathcal{I}_t$.) Note that this interpretation fixes each formula in $L$. 

More generally, if $\mathcal{S} \vdash \exists^! x \dt (\phi(x) \wedge x \in \mathbb{N})$, then $\mathcal{I}_\phi$ denotes the interpretation generated by interpreting $\Mod$ by $\forall x \in \mathbb{N} \dt (\phi(x) \rightarrow \Modrm_x)$ and $\Uni$ by $\forall x \in \mathbb{N} \dt (\phi(x) \rightarrow \Unirm_x)$. We may expand the language with a constant symbol $c_\phi$ and extend $\mathcal{S}$ with the axiom $\forall x \dt (\phi(x) \leftrightarrow x = c_\phi)$, to produce an interpretation $\mathcal{I}_{c_\phi}$ equivalent to $\mathcal{I}_\phi$.

Let $T_0$ and $T_1$ be theories in the languages $L_0$ and $L_1$, respectively. Let $\mathfrak{F}$ be a family (set) of interpretations from $L_0$ to $L_1$. We say that $\mathfrak{F}$ is a {\em local interpretation} of $T_0$ in $T_1$ if for any finite set $T'_0$ of consequences of $T_0$, there is an $\mathcal{I} \in \mathfrak{F}$ which interprets $T'_0$ in $T_1$. Alternatively, we say that $T_1$ locally interprets $T_0$ by $\mathfrak{F}$, respectively.

Given revision parameters $\Trm, \Unirm, \Modrm_0$, recall that $\lang^\Trm$ is the language of the theories $\Tb_n$ and that $\lang^\Rev$ is any sublanguage of $\lang^\Trm$.

\begin{mainlemma*}\label{Main Lemma: Reflective construction}
Let $\Trm$, $\Unirm$ and $\Modrm_0$ be revision parameters such that $\Tb_\omega \vdash \ZF$, let $\mathfrak{F} = \{\mathcal{I}_{\underline{k}} \mid k \in \mathbb{N}\}$ and let $S$ be an $\lang^\Rev_{\Uni, \Mod}$-theory such that for any finite $\Gamma \subseteq S$, 
\[
\exists A \in \mathbb{N} \dt \forall k \in \mathbb{N} \dt [A \leq k \rightarrow \Tb_k \vdash \mathcal{I}_{\underline{k}}^{\Unirm, \Modrm_0}(\Gamma)].
\]
\begin{enumerate}[{\rm (a)}]
\item If $\Trm, \Unirm, \Modrm_0$ admit $\NEC^*$, then $\Tb_\omega$ locally interprets $\CM^- + \NEC + S$ by $\mathfrak{F}$.
\item If $\Trm, \Unirm, \Modrm_0$ admit $\CONEC^*$, then $\Tb_\omega$ locally interprets $\CM^- + \CONEC + S$ by $\mathfrak{F}$.
\item If $\Trm, \Unirm, \Modrm_0$ admit $\NEC^*$ and $\CONEC^*$, then $\Tb_\omega$ locally interprets $\MS^- + S$ by $\mathfrak{F}$.
\end{enumerate}
\end{mainlemma*}
\begin{proof}
We prove the latter, most complicated assertion; the other two assertions follow by restricting the proof to the appropriate cases. Assume that $\Trm, \Unirm, \Modrm_0$ are revision parameters admitting $\NEC^*$ and $\CONEC^*$. 

Assume that $\MS^- + S \vdash \psi$. By Lemma \ref{Lemma: Antons lemma}, we have that $\MS^-_\mathrm{NC} + S \vdash \psi$. Let $\rho$ be a linear Hilbert-style proof witnessing this. Let $\Gamma$ be the axioms of $\CM^- + S$ occurring in $\rho$. We shall start by showing that
\[
\exists A' \in \mathbb{N} \dt \forall k \in \mathbb{N} \dt [A' \leq k \rightarrow \Tb_k \vdash \mathcal{I}_{\underline{k}}^{\Unirm, \Modrm_0}(\Gamma)]. \tag{*}
\]

We have by the assumption of the lemma that there is $A < \omega$, such that for any $k < \omega$ with $A \leq k$, we have $\Tb_k \vdash \mathcal{I}^{\Uni, \Mod_0}_{\underline{k}}(\phi)$, for every axiom $\phi$ of $S$ in $\Gamma$. Moreover, note that for any $k < \omega$, and any axiom $\phi$ of $\ZF + \Sep(\lang_{\Uni, \Mod}) + \Rep(\lang_{\Uni, \Mod})$, $\mathcal{I}^{\Uni, \Mod_0}_{\underline{k}}(\phi)$ is an axiom of $\ZF$. So since $\Tb_\omega \vdash \ZF$, there is $A' \geq 1$ with $A \leq A' < \omega$, such that for any $k < \omega$ with $A' \leq k$, we have $\Tb_{k} \vdash \mathcal{I}^{\Uni, \Mod_0}_{\underline{k}}(\phi)$, for every axiom $\phi$ of $\ZF + \Sep(\lang_{\Uni, \Mod}) + \Rep(\lang_{\Uni, \Mod})$ in $\Gamma$.

Suppose that $\phi$ is a compositional axiom (of the form $\CM_-$). Then $\forall A' \leq k < \omega \dt [\Tb_k \vdash \mathcal{I}^{\Unirm, \Modrm_0}_{\underline{k}}(\phi)]$ follows from that for all $\lang^\Trm$-structures $\mathcal{U}$ and all $\lang^\Rev_{\Uni, \Mod}$-formulas $\phi$, it is provable that for all $A' \leq k < \omega$,
\[
\Modrm_{{k}}(\mathcal{U}, \phi) \iff \langle \mathcal{U}\restr_{\lang^\Rev}, \gquote{\Unirm_{{k-1}}}^\mathcal{U}, \gquote{\Modrm_{{k-1}}}^\mathcal{U} \rangle \models \phi
\]
(using $1 \leq A'$), and from that the corresponding compositional conditions hold for $\models$. 

Hence, $A'$ satisfies (*) as desired. We introduce shifted parameters $\Trm', \Unirm', \Modrm'_0$ defined by $\Trm'_k \equiv_\df \Trm_{k+\underline{A'}}$, $\Unirm'_k \equiv_\df \Unirm_{k+\underline{A'}}$ and $\Modrm'_k \equiv_\df \Modrm_{k+\underline{A'}}$. Note that $\Modrm'$ also satisfies ($\dagger$) in Construction \ref{Constr: Multiverse semantics}. It is easily seen that $\Trm', \Unirm', \Modrm'_0$ are revision parameters admitting $\NEC^*$ and $\CONEC^*$, and that $\Tb'_\omega = \Tb_\omega$. Note that for all $\phi$, it is provable that for all $k < \omega$, $\mathcal{I}^{\Unirm', \Modrm'_0}_{\underline{k}}(\phi) \leftrightarrow \mathcal{I}^{\Unirm, \Modrm_0}_{\underline{k+A'}}(\phi)$. 

We index the sequence of steps in the proof by numbers $0, 1, \cdots, l-1$, where $l$ is the length of $\rho$. For each $q < l$, let $\psi_q$ be the derived formula (or axiom) at step $q$ of $\rho$ (so $\psi = \psi_{l-1}$), let $N_q$ be the number of applications of $\NEC$ in the derivation of $\psi_q$, and let $C_q$ be the number of applications of $\CONEC$ in the derivation of $\psi_q$. It suffices to show that there are natural numbers $m, n$, such that $\Tb'_m \vdash \mathcal{I}^{\Unirm', \Modrm'_0}_{\underline{n}}(\psi_q)$, for each $q < l$. We do so by induction on the steps of $\rho$. Let $r < l$. Here is our induction hypothesis: 

\begin{itemize}
\item[(IH)] $\Tb'_{2C_q + k} \vdash \mathcal{I}^{\Unirm', \Modrm'_0}_{\underline{k}}(\psi_q)$, for any step $q < r$, and for any $k \geq N_q + 1$.
\end{itemize}

We need to show that $\Tb'_{2C_r + k} \vdash \mathcal{I}^{\Unirm', \Modrm'_0}_{\underline{k}}(\psi_r)$, whenever $N_r + 1 \leq k$. There are four cases, as to which rule of inference (if any) is applied to obtain $\psi_r$ from $\{\psi_q \mid q < r\}$:

\begin{description}
\item[(Axiom)] $\psi_r$ is an axiom in $\Gamma$.
\item[(First-order)] $\psi_r$ is derived by a rule of inference of first-order logic.
\item[(NEC)] $\psi_r$ is derived from $\psi_{r'}$ by $\NEC$, for some $r' < r$.
\item[(CONEC)] $\psi_r$ is derived from $\psi_{r'}$ by $\CONEC$, for some $r' < r$.
\end{description}
Let $k \geq N_r + 1$. We proceed to show $\Tb'_{2C_r + k} \vdash \mathcal{I}^{\Unirm', \Modrm'_0}_{\underline{k}}(\psi_r)$ in each of the above cases. The fact that $\Tb'_a \vdash \Tb'_b$, for any $0 \leq b \leq a \in \mathbb{N}$, will be used repeatedly without mention.

\begin{description}
\item[(Axiom)] Suppose that $\psi_r$ is an axiom in $\Gamma$. We have $\Tb'_k \vdash \mathcal{I}^{\Unirm', \Modrm'_0}_{\underline{k}}(\phi_r)$, since $\Tb'_k \vdash \Tb_{k+A'}$ and $\Tb_{k+A'} \vdash \mathcal{I}^{\Unirm, \Modrm_0}_{\underline{k+A'}}(\phi_r)$.

\item[(First-order)] By (IH), $\Tb'_{2C_r + k} \vdash \mathcal{I}^{\Unirm', \Modrm'_0}_{\underline{k}}(\phi_q)$, for each $q < r$. Since $\mathcal{I}^{\Unirm', \Modrm'_0}_{\underline{k}}$ is an interpretation, it respects the inference rules of first-order logic. Therefore, $\Tb'_{2C_r + k} \vdash \mathcal{I}^{\Unirm', \Modrm'_0}_{\underline{k}}(\psi_r)$.

\item[(NEC)] In this case $\psi_r$ is $\forall \mathcal{U} \in \Uni\dt (\Mod(\mathcal{U}, \gquote{\psi_{r'}}))$. Note that $C_r = 0$ and $N_{r'} < N_r$.
\begin{align*}
\Tb'_{k-1} &\vdash \mathcal{I}^{\Unirm', \Modrm'_0}_{\underline{k-1}}(\psi_{r'}) & & \text{(IH)} \\
\Tb'_{k} &\vdash \forall \mathcal{U} \in \Unib'_{\underline{k}} \dt ( \mathcal{U} \models \gquote{\mathcal{I}^{\Unirm', \Modrm'_0}_{\underline{k-1}}(\psi_{r'})}) & & \NEC^* \\
\Tb'_{k} &\vdash \forall \mathcal{U} \in \Unib'_{\underline{k}} \dt (\Modrm'_{\underline{k}}(\mathcal{U}, \gquote{\psi_{r'}})) & & \text{($\dagger$) in Construction \ref{Constr: Multiverse semantics}} \\
\Tb'_{k} &\vdash \mathcal{I}^{\Unirm', \Modrm'_0}_{\underline{k}} \big(\forall \mathcal{U} \in \Uni\dt (\Mod(\mathcal{U}, \gquote{\psi_{r'}})) \big) & & \text{Definition of } \mathcal{I}^{\Unirm', \Modrm'_0}_{\underline{k}}
\end{align*}

\item[(CONEC)] In this case $\psi_{r'}$ is $\forall \mathcal{U} \in \Uni\dt (\Mod(\mathcal{U}, \gquote{\psi_{r}}))$.
\begin{align*}
\Tb'_{2C_{r'} + k+1} &\vdash \mathcal{I}^{\Unirm', \Modrm'_0}_{\underline{k+1}} \big(\forall \mathcal{U} \in \Uni\dt (\Mod(\mathcal{U}, \gquote{\psi_r}))\big) & & \text{(IH)} \\
\Tb'_{2C_{r'} + k+1} &\vdash \forall \mathcal{U} \in \Unib'_{\underline{k + 1}} \dt (\Modrm'_{\underline{k + 1}}(\mathcal{U}, \gquote{\psi_r})) & & \text{Definition of } \mathcal{I}^{\Unirm', \Modrm'_0}_{\underline{k+1}} \\
\Tb'_{2C_{r'} + k+1} &\vdash \forall \mathcal{U} \in \Unib'_{\underline{k + 1}} \dt (\mathcal{U} \models \gquote{\mathcal{I}^{\Unirm', \Modrm'_0}_{\underline{k}}(\psi_r)}) & & \text{($\dagger$) in Construction \ref{Constr: Multiverse semantics}} \\
\Tb'_{2C_{r'} + k+1} &\vdash \forall \mathcal{U} \in \Unib'_{\underline{2C_{r'} + k+1}} \dt (\mathcal{U} \models \gquote{\mathcal{I}^{\Unirm', \Modrm'_0}_{\underline{k}}(\psi_r)}) & & \Unib'_{\underline{2C_{r'} + k+1}} \subseteq \Unib'_{\underline{k + 1}} \\
\Tb'_{2C_{r'} + k+2} &\vdash \mathcal{I}^{\Unirm', \Modrm'_0}_{\underline{k}}(\psi_r) & & \CONEC^* \\
\Tb'_{2C_r + k} &\vdash \mathcal{I}^{\Unirm', \Modrm'_0}_{\underline{k}}(\psi_r) & & C_{r'} < C_r
\end{align*}
\end{description}

This completes the proof of the lemma. We close by recording the more detailed statement that we have actually proved: If $N,C$ are the number of applications of $\NEC$ and $\CONEC$, respectively, in a proof of $\psi$ in $\MS^-_\mathrm{NC} + S$, then 
\[
\Tb'_{2C + N + 1 + A'} \vdash \mathcal{I}^{\Unirm', \Modrm'_0}_{\underline{N + 1 + A'}}(\psi).
\]
\end{proof}

\begin{cor}\label{Cor: T interprets MS}
Let $\Trm$, $\Unirm$ and $\Modrm_0$ be revision parameters such that for some $B \in \mathbb{N}$, $\Tb_B \vdash \ZF$. Let $\mathfrak{F} = \{\mathcal{I}_{\underline{k}} \mid k \in \mathbb{N}\}$ and let $S$ be a theory (possibly with $\Uni, \Mod$ in its language) such that for any finite $\Gamma \subseteq S$, 
\[
\exists A \in \mathbb{N} \dt \forall k \in \mathbb{N} \dt [A \leq k \rightarrow \Tb_k \vdash \mathcal{I}_{\underline{k}}^{\Unirm, \Modrm_0}(\Gamma)].
\]
\begin{enumerate}[{\rm (a)}]
\item If $\Trm, \Unirm, \Modrm_0$ admit $\NEC^*$, then $\Tb_\omega$ locally interprets $\CM + \NEC + S$ by $\mathfrak{F}$.
\item If $\Trm, \Unirm, \Modrm_0$ admit $\CONEC^*$, then $\Tb_\omega$ locally interprets $\CM + \CONEC + S$ by $\mathfrak{F}$.
\item If $\Trm, \Unirm, \Modrm_0$ admit $\NEC^*$ and $\CONEC^*$, then $\Tb_\omega$ locally interprets $\MS + S$ by $\mathfrak{F}$.
\end{enumerate}
\end{cor}
\begin{proof}
Let $B \in \mathbb{N}$, such that $\Tb_B \vdash \ZF$. Since $\Trm$, $\Unirm$ and $\Modrm_0$ are revision parameters, we have 
\[
\forall k \in \mathbb{N} \dt [B+1 \leq k \rightarrow \Tb_k \vdash \forall \mathcal{U} \in \Unib_{\underline{k}} \dt (\mathcal{U} \models \udot{\ZF})].
\]
So by the definition of $\mathcal{I}^{\Unirm, \Modrm_0}$ and ($\dagger$) in Construction \ref{Constr: Multiverse semantics},
\[
\forall k \in \mathbb{N} \dt [B+1 \leq k \rightarrow \Tb_k \vdash \mathcal{I}_{\underline{k}}^{\Unirm, \Modrm_0}(\mathsf{Multiverse}_\ZF)].
\]
Applying the Main Lemma with $S + \mathsf{Multiverse}_\ZF$ for $S$, and the maximum of $A$ and $B+1$ for $A$, we obtain the desired result.
\end{proof}

The following systems, along with $\GR^\alpha$ from Definition \ref{Sys: GR omega}, are useful for measuring the consistency strength of various extensions of $\CM$.

\begin{sys}\label{Sys: Reflection}
Let $S$ be a set-theoretic system. For any set theory $T$ in language $L$, $\R_T$ is the so called {\em (proof-theoretic) Reflection schema}:
\[
\begin{array}{ll}
\RTh_{\udot T} & \{ {\Prv_{\udot T}}(\gquote{\phi}) \rightarrow \phi \mid \phi \in L\}.
\end{array}
\]
(The dot under $T$ is sometimes omitted, when it is clear from the context.) 

We recursively define, for recursive ordinals\footnote{An ordinal $\alpha$ is {\em recursive} if there is a $\Sigma_1^0$-formula defining a well-ordering of a subset of $\mathbb{N}$ of order-type $\alpha$. These are precisely the ordinals below $\omega_1^\mathrm{CK}$.} $\alpha$, the theories $\ConTh^\alpha(S)$ and $\RTh^\alpha(S)$, of {\em $\alpha$-iterated Consistency over $S$} and {\em $\alpha$-iterated Reflection schema over $S$}, respectively:
\begin{align*}
\ConTh^0(S) &=_\df S  \\
\ConTh^{\alpha+1}(S) &=_\df \ConTh^\alpha(S) + \Con_{\ConTh^\alpha(S)}  \\
\ConTh^\alpha(S) &=_\df \bigcup_{\xi < \alpha} \ConTh^\xi(S) \text{, for $\alpha$ a limit ordinal}; \\
\RTh^0(S) &=_\df S \\
\RTh^{\alpha+1}(S) &=_\df \RTh^\alpha(S) + \R_{\RTh^\alpha(S)} \\
\RTh^\alpha(S) &=_\df \bigcup_{\xi < \alpha} \RTh^\xi(S) \text{, for $\alpha$ a limit ordinal.}
\end{align*}
We use the notations $\ConTh^\alpha$ and $\RTh^\alpha$ for $\ConTh^\alpha(\ZF)$ and $\RTh^\alpha(\ZF)$, respectively. 
\end{sys}

Recall System \ref{Sys: GR omega}, where $\GR^\alpha$ is defined, using the axiom $\GR_T$ of Global Reflection over a set theory $T$ extending $\CT\hspace{-3pt}\restriction$ (in some language $L_\Sat$ with a satisfaction predicate):
\[
\forall \phi \in \udot{L}_\Sat \dt (\Prv_{\udot T}(\phi) \rightarrow \Tr(\phi))
\]
Comparing $\RTh_T$ with $\GR_T$, note that $\RTh_T$ is a schema, with a separate axiom for each formula of the form $\phi$ in the meta-language, while $\GR_T$ quantifies internally over all formulas in the object-language; the latter is made possible by the satisfaction/truth-predicate.

\begin{rem*} Let us pause to measure the consistency strengths of $\GR^\omega$, $\mathsf{R}^\omega$ and $\mathsf{Con}^\omega$: The consistency strength of $\GR^\omega$ is bounded by that of $\mathsf{MK} + \mathsf{GC}$ (Morse-Kelley class theory with Global Choice),\footnote{The following argument indicates that the consistency strength of $\GR^\omega$ is far less than that of $\mathsf{MK} + \mathsf{GC}$. We rely on \cite{Fuj12}: Fujimoto shows in his Theorem 70 that the consistency strength of his theory of truth, $\FS$, is equal that of $\mathsf{NBG}_\omega$, which is a subtheory of $\mathsf{NBG}_{< E_0}$ introduced in \cite{JK10}. By Theorem 15 in (ibid.), and by Fujimoto's Proposition 4, $\mathsf{NBG}_{< E_0}$ is a subsystem of $\mathsf{MK} + \mathsf{GC}$. (All of the relevant definitions are found in Fujimoto's \S 3.1.) Moreover, the proof of Fujimoto's Proposition 21 provides the base step, and $\NEC$ provides the induction step, to show that his $\FS$ proves the version of $\GR^\omega$ for truth (even with the Replacement schema extended to the language with the truth predicate). Since that version of $\GR^\omega$ interprets our $\GR^\omega$, the consistency strength of our $\GR^\omega$ is bounded by the consistency strength of $\mathsf{NBG}_\omega$ (and since our $\GR^\omega$ does not have the Replacement schema extended to the language with the satisfaction relation, this bound is probably not tight), which in turn is bounded by the consistency strength of $\mathsf{MK} + \mathsf{GC}$.} which, in turn, is far less than that of $\ZFC + \text{``there exists an inaccessible cardinal''}$.\footnote{If $\kappa$ is an inaccessible cardinal, then $V_{\kappa+1}$ provides a natural model of $\mathsf{MK} + \mathsf{GC}$.} The consistency strength of $\mathsf{R}^\omega$ is bounded by that of $\GR^1$.\footnote{This is shown by a routine argument, using Global Reflection to prove each iteration of the Reflection schema.} The consistency strength of $\mathsf{Con}^\omega$ is bounded by that of $\mathsf{R}^1$.\footnote{This is shown by a routine argument, using the Reflection schema to prove each iteration of Consistency.} Moreover, it is easily observed that for each $n \in \mathbb{N}$: $\GR^{n+1}$, $\mathsf{R}^{n+1}$ and $\mathsf{Con}^{n+1}$ proves the consistency of $\GR^{n}$, $\mathsf{R}^{n}$ and $\mathsf{Con}^{n}$, respectively.
\end{rem*}

\begin{thm}\label{Thm: Con^omega interprets NEC + Non-triv}
$\ConTh^\omega$ locally interprets $\CM + \NEC + \textnormal{\sf Non-Triviality}$.
\end{thm}
\begin{proof}
We can choose revision parameters $\Trm, \Unirm, \Modrm_0$, such that for each $n \in \mathbb{N}$:
\begin{align*}
\Tb_n &= \ConTh^n \\
\Unib_{n+1} &= \{ \mathcal{U} \mid \mathcal{U} \models \Tb_n \}
\end{align*}
Clearly, these are revision parameters provably satisfying $\textnormal{\sf Soundness}^*$, and thereby admitting $\NEC^*$. Moreover, we have for each $k \in \mathbb{N}$:
\begin{align*}
\Tb_{k+1} &\vdash \Unib_{\underline{k+1}} \neq \varnothing & & \text{Definition of } \Tb_{k+1},  \Unib_{k+1} \\
\Tb_{k+1} &\vdash \mathcal{I}^{\Unirm, \Modrm_0}_{\underline{k+1}}(\textnormal{\sf Non-Triviality}) & & \text{Definition of } \mathcal{I}^{\Unirm, \Modrm_0}_{\underline{k+1}}
\end{align*}
So the result follows from Corollary \ref{Cor: T interprets MS}, by setting $S = \{\textnormal{\sf Non-Triviality}\}$.
\end{proof}

\begin{rem*}
Under the mild meta-theoretic assumption that each $\ConTh^n$ is closed under the {\sf Reflection rule}, it follows from Lemma \ref{Lemma: Admissible rules on revision parameters} that the revision parameters in the above proof admit $\CONEC^*$, yielding that $\ConTh^\omega$ locally interprets $\CM + \NEC + \textnormal{\sf Non-Triviality}$. This meta-theoretic assumption follows from the assumption that $\ConTh^\omega$ is $\omega$-consistent, which in turn follows from the existence of an $\omega$-standard model of $\ZF$.
\end{rem*}
\begin{rem*}
Using the fine-grained result obtained at the end of the proof of the Main Lemma, we can show by an overspill-argument that $\ConTh^\omega + \{\underline{n} < c \mid n \in \mathbb{N}\}$ (for a fresh constant $c$) interprets $\CM + \NEC + \textnormal{\sf Non-Triviality}$ (not just locally). This raises:
\end{rem*}
\begin{que*}
Is $\CM + \NEC + \textnormal{\sf Non-Triviality}$ $\omega$-inconsistent?
\end{que*}

\begin{thm}\label{Thm: R^omega interprets T}
$\RTh^\omega$ locally interprets $\MS + \textnormal{\sf Multiverse Reflection}$.
\end{thm}
\begin{proof}
We can choose revision parameters $\Trm, \Unirm, \Modrm_0$, such that for each $n \in \mathbb{N}$:
\begin{align*}
\Tb_n &= \RTh^n \\
\Unib_{n+1} &= \{ \mathcal{U} \mid \mathcal{U} \models \Tb_n \}
\end{align*}
It is easily seen that these are revision parameters provably satisfying $\textnormal{\sf Soundness}^*$, and thereby admitting $\NEC^*$. Similarly, it is easily seen that they admit the $\textnormal{\sf Reflection rule}^*$ and satisfy $\textnormal{\sf Completeness}^*$, so that they admit $\CONEC^*$. Moreover, we have for each $k \in \mathbb{N}$ and each $\phi \in \lang$:
\begin{align*}
\Tb_{k+1} &\vdash \big( \forall \mathcal{U} \in \Unib_{\underline{k+1}} \dt \Modrm_{\underline{k+1}}(\mathcal{U}, \gquote{\phi}) \big) \rightarrow \phi & & \text{The completeness theorem and} \\
& & & \text{the definition of } \Tb_{k+1},  \Unib_{k+1} \\
\Tb_{k+1} &\vdash \mathcal{I}^{\Unirm, \Modrm_0}_{\underline{k+1}}(\Tr^\Box(\gquote{\phi}) \rightarrow \phi) & & \text{Definition of } \mathcal{I}^{\Unirm, \Modrm_0}_{\underline{k+1}}
\end{align*}
So the result follows from Corollary \ref{Cor: T interprets MS}, by setting $S ={\textnormal{\sf Multiverse Reflection}}$.
\end{proof}

\begin{rem*}
The technique for obtaining full (not just local) interpretability, mentioned in the second remark following Theorem \ref{Thm: Con^omega interprets NEC + Non-triv}, does not work for the above theorem, because {\sf Multiverse Reflection} is not finitely axiomatizable (as far as the authors can see). The overspill-argument must be carried out on a single formula, not a schema.
\end{rem*}

\begin{thm}\label{Thm: multiverse conservativity}
The following conservativity results hold:
\begin{enumerate}[{\rm (a)}]
\item $\mathsf{Comp}_\CM \equiv_\lang \ZF$
\item $\CM + \NEC + \textnormal{\sf Non-Triviality} \equiv_\lang \ConTh^\omega$ 
\item $\MS + \textnormal{\sf Multiverse Reflection} \equiv_\lang \CM + \NEC + \textnormal{\sf Multiverse Reflection} \equiv_\lang \RTh^\omega$
\end{enumerate}
\end{thm}
\begin{proof}
\begin{enumerate}[{\rm (a)}]
\item This is immediate from Theorem \ref{Thm: GL interpretation}(\ref{Item: ZF interprets Comp}) and that the interpretation used in its proof restricts to the identity on $\lang$.
\item The right-to-left direction follows from Theorem \ref{Thm: Con^omega interprets NEC + Non-triv}, observing that the interpretation is the identity on $\lang$.
 
For the left-to-right direction, suppose as induction hypothesis that we have proved $\ConTh^n$ in $\CM +\NEC + \textnormal{\sf Non-Triviality}$. By $\NEC$, we have $\Tr^\Box(\udot{\ConTh^{\underline{n}}})$. So by {\sf Non-Triviality} and Lemma \ref{Lemma: CM soundness}, we can prove $\ConTh^n + \Con(\udot\ConTh^{\underline{n}})$, which is $\ConTh^{n+1}$, as desired.

\item That $\RTh^\omega$ proves every $\lang$-theorem of $\MS + \textnormal{\sf Multiverse Reflection}$ follows from Theorem \ref{Thm: R^omega interprets T}, observing that the interpretation is the identity on $\lang$. 

That $\MS + \textnormal{\sf Multiverse Reflection} \vdash \CM + \NEC + \textnormal{\sf Multiverse Reflection}$ is trivial.

For $\CM + \NEC + \textnormal{\sf Multiverse Reflection} \vdash \RTh^\omega$, we shall show that for each $n \in \mathbb{N}$, $\CM + \NEC + \textnormal{\sf Multiverse Reflection} \vdash \Tr^\Box(\udot{\RTh}^{\underline{n}})$. Then the result follows from $\textnormal{\sf Multiverse Reflection}$. We proceed by induction. By $\textnormal{\sf Multiverse}_\ZF$, we have the base case: $\MS + \textnormal{\sf Multiverse Reflection} \vdash \Tr^\Box(\udot{\RTh}^{\underline{0}}).$ So suppose as induction hypothesis that 
\[
\MS + \textnormal{\sf Multiverse Reflection} \vdash \Tr^\Box(\udot{\RTh}^{\underline{k}}).
\]
Let $\sigma \in \sent(\lang)$. By $\NEC$ we have
\begin{align}
\MS + \textnormal{\sf Multiverse Reflection} &\vdash \Tr^\Box(\gquote{ \Tr^\Box(\udot{\RTh}^{\underline{k}}) }) \label{Eq: 1} \\
\MS + \textnormal{\sf Multiverse Reflection} &\vdash \Tr^\Box(\udot{\CM}) \label{Eq: 2} \\
\MS + \textnormal{\sf Multiverse Reflection} &\vdash \Tr^\Box(\gquote{\Tr^\Box(\gquote{\sigma}) \rightarrow \sigma}), \label{Eq: 3}
\end{align}
since (by {\sf Multiverse}$_\ZF$) $\NEC$ only needs to be applied to finitely many axioms of $\CM$. 
We work in $\MS + \textnormal{\sf Multiverse Reflection}$. Let $\mathcal{U} \in \Uni$. Assume $\Mod(\mathcal{U}, \gquote{\Pr_{\udot{\RTh}^{\underline{k}}}(\gquote{\sigma})}).$ We shall show $\Mod(\mathcal{U}, \gquote{\sigma})$. By (\ref{Eq: 2}), we can apply the Soundness Lemma in $\mathcal{U}$ to obtain $\Mod(\mathcal{U}, \gquote{\Tr^\Box(\gquote{\sigma})})$ from (\ref{Eq: 1}) and $\Mod(\mathcal{U}, \gquote{\Pr_{\udot{\RTh}^{\underline{k}}}(\gquote{\sigma})})$. Now by $\CM_\rightarrow$ and (\ref{Eq: 3}), we obtain $\Mod(\mathcal{U}, \gquote{\sigma})$, as desired.
\qedhere
\end{enumerate}
\end{proof}

\begin{thm}\label{Thm: CM + NEC + self-iso}
$\GR^\omega$ interprets $\CM(\lang_{\iota, \self}) + \NEC + \textnormal{\sf Self-Perception}$. 
\end{thm}
\begin{proof}
Set $\lang^\Trm = \lang_{\Sat, \iota, \self}$ and set $\lang^\Rev$ to $\lang_{\iota, \self}$. We can choose revision parameters $\Trm, \Unirm, \Modrm_0$, such that for each $n \in \mathbb{N}$:
\begin{align*}
\Tb_n &= \SP^n \\
\Unib_{n+1} &= \{\mathcal{U} \mid \mathcal{U} \models \Tb_{n} \wedge \mathcal{U} \in \crsmb\}
\end{align*}
Clearly these are revision parameters satisfying $\textnormal{\sf Soundness}^*$, thus admitting $\NEC^*$. Moreover, note that for each $k < \omega$, $\Tb_k \vdash \Unirm_{\underline{k}}(\self)$. So for each $k < \omega$, 
\[
\Tb_k \vdash \mathcal{I}^{\Unirm, \Modrm_0}_{\underline{k}}(\textnormal{\sf Self-Perception}).
\]
Thus, it follows from Corollary \ref{Cor: T interprets MS} that $\mathfrak{F} = \{\mathcal{I}_{\underline{k}} \mid k < \omega\}$ is a local interpretation of $\CM(\lang_{\iota, \self}) + \NEC + \textnormal{\sf Self-Perception}$ in $\SP^\omega$.

Recall from Lemma \ref{Lemma: GR^omega interprets SP^omega} that there is an interpretation $\mathcal{J}$ of $\SP^\omega$ in $\GR^\omega$. For each $n < \omega$, let $\mathcal{K}_{\underline{n}} = \mathcal{J} \circ \mathcal{I}_{\underline{n}}$. Then $\mathfrak{G} = \{\mathcal{K}_{\underline{k}} \mid k < \omega\}$ is a local interpretation of $\CM(\lang_{\iota, \self}) + \NEC + \textnormal{\sf Self-Perception}$ in $\GR^\omega$. There is a technical hurdle later on in this proof caused by the fact that the image of each $\mathcal{K}_{\underline{n}}$ is not included in $\lang$, but includes the instance of $\Tr$ needed to define the set $\mathbf{Tr} = \{\sigma \in \sent(\lang_\Sat) \mid \Tr(\sigma)\}$ for the construction of $\mathcal{J}$ in the proof of Lemma \ref{Lemma: GR^omega interprets SP^omega}. To overcome this, we construct the functions $\mathcal{K}'_{\underline{n}}(\phi)$, for each $n < \omega$, replacing each occurrence of the form ``$\Tr(t)$'' in the values of $\mathcal{K}_{\underline{n}}$ by ``$t \in y$'', where $y$ is assumed to be fresh. Then we have for each $\phi$ that 
\[
\GR^\omega \vdash \mathcal{K}(\phi) \leftrightarrow (\mathcal{K}'(\phi))[\mathbf{Tr} / y].
\]

Let $f$ be a function enumerating the theorems of $\CM(\lang_{\iota, \self}) + \NEC + \textnormal{\sf Self-Perception}$, such that the length of the theorems is non-decreasing. Then, for each $i < \omega$ there is $\sigma$, such that $\GR^\omega \vdash \udot{f}(\underline{i}) = \gquote{\sigma}$. Let $\theta(x, z)$ be the formula expressing
\[
x < \omega \wedge \exists j < \omega \dt \forall i \leq x \dt \Sat \big( \udot{\mathcal{K}}'_{\underline{\underline{j}}}(f(i)), \gquote{y} \mapsto z \big).
\]
Since the values of the $\mathcal{K}'_{\underline{n}}$ are in $\lang$, we have by Proposition \ref{Prop: T biconditionals} that for each $j < \omega$ and each $\phi \in \lang^\Rev_{\Uni, \Mod}$:
\[
\GR^\omega \vdash \Sat(\udot{\mathcal{K}}'_{\underline{\underline{j}}}(\gquote{\phi}), \gquote{y} \mapsto z) \leftrightarrow \mathcal{K}'_{\underline{j}}(\phi)[z / y].
\]
So since $\mathfrak{G}$ is a local interpretation, we have for each $n < \omega$ that
\[
\GR^\omega \vdash \theta(\underline{n}, \mathbf{Tr}).
\]
Since $\GR^\omega$ is $\omega$-inconsistent, there is a formula $\uranus(x)$ such that $\GR^\omega \vdash \exists x < \omega \dt\uranus(x)$, but for each $n \in \mathbb{N}$, $\GR^\omega \vdash \neg\uranus(\underline{n})$. Working in $\GR^\omega$, employing Proposition \ref{Prop: CT induction}, we obtain a maximal number $d < \omega$ such that $\theta(d, \mathbf{Tr}) \wedge \neg\uranus(d)$. Now there is a minimal number $e < \omega$ such that
\[
\forall i \leq d \dt \Sat \big( \udot{\mathcal{K}}'_{\underline{\underline{e}}}(f(i)), \gquote{y} \mapsto \mathbf{Tr} \big).
\]
So $\GR^\omega$ defines $e$ by a formula $\psi$. It follows from Proposition \ref{Prop: T biconditionals} that $\mathcal{K}_\psi$ is an interpretation of $\CM(\lang_{\iota, \self}) + \NEC + \textnormal{\sf Self-Perception}$ in $\GR^\omega$.
\end{proof}

\begin{rem*}
Under the meta-theoretic assumption that the revision parameters used in the above proof admit the {\sf Reflection-rule*}, $\GR^\omega$ interprets $\MS(\lang_{\iota, \self}) + \textnormal{\sf Self-Perception}$. Another potential approach to validating $\CONEC$ would be to employ revision parameters based on a hierarchy of theories converging to $\FS\restr {} + \Sep(\lang_\Sat)$, rather than to $\GR^\omega$ (but the authors are not certain that this approach would work). 
\end{rem*}

\begin{que*}
Does $\FS\restr {} + \Sep(\lang_\Sat)$ interpret $\MS(\lang_{\iota, \self}) + \textnormal{\sf Self-Perception}$?
\end{que*}

\begin{rem*}
There is another proof of the above theorem that utilizes the fine-grained result at the end of the proof of the Main Lemma. That technique only works when the theory $S$ of the Main Lemma is finitely axiomatized (as here where $S = \{\textnormal{\sf Self-Perception}\}$). But the technique used in the above proof works also for non-finitely axiomatized $S$.
\end{rem*}

\begin{rem*} Note the contrast that $\CM$ includes $\ZF + \Sep(\lang_{\Uni, \Mod})  + \Rep(\lang_{\Uni, \Mod})$ while $\GR^\omega$ only includes $\ZF + \Sep(\lang_\Sat)$. The essential reason why $\GR^\omega$ interprets $\Rep(\lang_{\Uni, \Mod})$ is that the interpretations $\mathcal{I}_{\underline{n}}$ in the above proof map $\Mod$ and $\Uni$ to the $\lang$-formulas $\Modrm_{\underline{n}}$ and $\Unirm_{\underline{n}}$, respectively.
\end{rem*}

In light of Theorems \ref{Thm: multiverse conservativity} and \ref{Thm: CM + NEC + self-iso}, the authors ask:

\begin{que*}
Is it the case that $\GR^\omega \equiv_\lang \CM(\lang_{\iota, \self}) + \NEC + \textnormal{\sf Self-Perception}$? If not, what is the precise consistency strength of $\CM(\lang_{\iota, \self}) + \NEC + \textnormal{\sf Self-Perception}$?
\end{que*}

The fact that the extended schema $\Sep(\lang_\Sat)$ in $\GR^\omega$ was only used to obtain the set of true {\em sentences}, suggests a negative answer; $\GR^\omega$ may have strictly higher consistency strength than $\CM(\lang_{\iota, \self}) + \NEC + \textnormal{\sf Self-Perception}$.

\section{Case studies}

This section examines how the framework introduced above can be applied to two rather different conceptions of the set-theoretic multiverse.

\subsection{Multiverse conceptions of arithmetical absoluteness}

One may feel confident in adopting a universe view on arithmetic, appealing to the general acceptance of an intended model consisting of the finite ordinals, while having a multiverse view of set theory, where agreement on an intended model is lacking. This subsection therefore explores how the techniques of this paper may be applied to a conception of the multiverse where arithmetic is more or less fixed throughout the universes.

Let $\lang^\PA$ be the language of arithmetic, and let $\Sigma_n^\PA$ be the usual complexity hierarchy of arithmetic formulas over $\PA$. Given $\phi \in \lang^\PA$, $\phi^\mathbb{N}$ denotes the corresponding $\lang$-formula obtained by restricting the quantifiers to $\mathbb{N}$. Figure \ref{Fig: Arithmetic Absoluteness} exhibits axioms, of increasing strength, expressing the absoluteness of arithmetic. The strongest of these is {\sf Arithmetic Absoluteness}, which expresses that the bounded quantifier ``$\forall x \in \mathbb{N}$'' commutes with the $\Mod$-relation, for the untyped language $\lang_{\Uni, \Mod}$.

\begin{figure}
\caption{Axioms of Arithmetic Absoluteness}
\label{Fig: Arithmetic Absoluteness}

\vspace{6pt}
\begin{center} 
{\bf Axioms of Arithmetic Absoluteness}
\end{center}
\[
\begin{array}{ll}
\textnormal{\sf $\Sigma^\PA_1$-Absoluteness} & \forall \sigma \in \sent(\Sigma_1^\PA):\dt \forall \mathcal{U} \in \Uni\dt \big( \Mod(\mathcal{U}, \gquote{\sigma^\mathbb{N}}) \leftrightarrow \sigma^\mathbb{N} \big) \\
\textnormal{\sf Arithmetic Absoluteness} & \forall \sigma \in \sent(\lang^\PA):\dt \forall \mathcal{U} \in \Uni\dt \big( \Mod(\mathcal{U}, \gquote{\sigma^\mathbb{N}}) \leftrightarrow \sigma^\mathbb{N} \big) \\
\textnormal{\sf Arithmetic Compositionality} & \forall \mathcal{U} \in \Uni\dt \forall \phi \in \lang_{\Uni, \Mod} \dt \big( \Mod(\mathcal{U}, \gquote{\forall x \in \mathbb{N}} \phi) \leftrightarrow \\ 
& \forall n \in \mathbb{N} \dt \Mod(\mathcal{U}, \phi, \gquote{x} \mapsto \underline{n}^\mathcal{U}) \big) \\
\end{array}
\]
\end{figure}

\begin{prop}\label{Prop: Arithm comp implies Arithm abs}
$\CM^- + \textnormal{\sf Arithmetic Compositionality} \vdash \textnormal{\sf Arithmetic Absoluteness}$
\end{prop}
\begin{proof}
Let $\sigma \in \sent(\lang^\PA)$, such that $\sigma^\mathbb{N}$ holds. Let $\mathcal{U} \in \Uni$. We show by induction on the syntactic structure that $\Mod(\mathcal{U}, \gquote{\sigma^\mathbb{N}}) \leftrightarrow \sigma^\mathbb{N}$. For atomic sentences it follows from that arithmetic equations are decidable. For the propositional connectives, the induction step follows from the axioms of the form $\CM_-$; let us look at $\sigma^\mathbb{N} \equiv \neg\phi$ for example:
\[ 
\Mod(\mathcal{U}, \gquote{\neg\phi}) \iff \neg \Mod(\mathcal{U}, \gquote{\phi}) \iff \neg\phi 
\]
The first equivalence holds by $\CM_\neg$ and the second by the induction hypothesis. For the quantifier case, suppose that $\sigma^\mathbb{N} \equiv \forall x \in \mathbb{N} \dt \phi(x)$. We calculate:
\begin{align*}
\Mod(\mathcal{U}, \gquote{\forall x \in \mathbb{N} \dt \phi(x)}) \iff \forall n \in \mathbb{N} \dt \Mod(\mathcal{U}, \gquote{\phi(\underline{n})}) \iff \forall n \in \mathbb{N} \dt \phi(\underline{n}) \iff \sigma^\mathbb{N}
\end{align*}
The first equivalence holds by Arithmetic Compositionality, the second by the induction hypothesis, and the third by the fact that for all $n \in \mathbb{N}$, $\mathbb{N} \models n = \underline{n}$. 
\end{proof}

The following proposition shows the reflective power of Arithmetic Absoluteness, and exhibits a scenario where $\CONEC$ is useful.

\begin{prop}\label{Prop: Arithm abs + CONEC implies Iterated Reflection}
$\CM + \textnormal{\sf $\Sigma^\PA_1$-Absoluteness} +\CONEC \vdash \RTh^{\omega^\mathrm{CK}_1}$
\end{prop}
\begin{proof}
By $\CONEC$ it suffices to prove $\Tr^\Box(\udot{\RTh}^{\omega^\mathrm{CK}_1})$. Naturally, we do so by transfinite induction. The base case follows from $\textnormal{\sf Multiverse}_\ZF$, an axiom of $\CM$. For the successor case, assume that $\Tr^\Box(\udot{\RTh}^\alpha)$ for some $\alpha < \omega^\mathrm{CK}_1$. Suppose $\mathcal{U} \in \Uni$, let $\sigma \in \lang$ and assume that $\Mod(\mathcal{U}, \gquote{\Pr_{\udot{\RTh}^\alpha}(\sigma)})$. Since $\alpha$ is recursive, $\Pr_{\udot{\RTh}^\alpha}$ is $\Sigma^0_1$, so by $\Sigma^\PA_1$-Absoluteness, we have $\Pr_{\udot{\RTh}^\alpha}(\sigma)$. It now follows from $\Tr^\Box(\udot{\RTh}^\alpha)$ and the Soundness Lemma that $\Mod(\mathcal{U}, \gquote{\sigma})$. So by $\CM_\rightarrow$, we have $\Mod(\mathcal{U}, \udot{\RTh}^{\alpha+1})$, as desired. The limit case is immediate from the definition of $\RTh^\alpha$. 
\end{proof}

\begin{rem*} Note that the above proof argues model-theoretically on an arbitrary universe. But even though $\CM + \textnormal{\sf $\Sigma^\PA_1$-Absoluteness} +\CONEC$ proves a fair amount of reflection, it is not clear to the authors whether it proves that there is a universe ({\sf Non-Triviality}).
\end{rem*}

This section raises questions about the consistency (strength) of combinations of Copernican multiverse theories and axioms of arithmetic absoluteness, in particular:

\begin{que*}
Is $\CM + \textnormal{\sf $\Sigma^\PA_1$-Absoluteness} +\CONEC$ consistent relative to $\RTh^{\omega^\mathrm{CK}_1}$? 
\end{que*}
\begin{que*}
Is $\CM^-(\lang_{\iota, \self}) + \NEC + \textnormal{\sf Self-Perception} + \textnormal{\sf Arithmetic Compositionality}$ consistent?
\end{que*}

The following Proposition may be viewed as a partial answer to the second question, but the authors do not consider it to suggest an ultimately negative answer:

\begin{prop}\label{Prop: Arithmetic compositionality gives omega-inconsistency}
The system $\CM^- + \NEC + \textnormal{\sf Arithmetic Compositionality} + \textnormal{\sf Non-Triviality}$ is $\omega$-inconsistent.
\end{prop}
\begin{proof}
This is a corollary of McGee's paradox, see \cite{McG85}.
\end{proof}

\subsection{The Hamkins multiverse}\label{Subsec: Hamkins's multiverse}

\cite{Ham12} introduced a conception of the set-theoretic multiverse, which informally is based on four over-arching principles:
\begin{enumerate}
\item The multiverse is a non-empty collection of models of $\ZFC$.
\item The multiverse is closed under the usual techniques for constructing models of set theory from other models of set theory, such as forcing extensions and inner models.
\item Every universe is countable and $\omega$-non-standard from the perspective of another universe.
\item The multiverse is closed under iterating large cardinal embeddings backward.
\end{enumerate}

Definition 1.1 in \cite[pp. 475--6]{GH10} gives succinct formulations of axioms encapsulating this conception. The third (and possibly the fourth) principle is more controversial than the first two. For example, the Well-foundedness Mirage axiom states that for every universe $\mathcal{U}$ there is a universe $\mathcal{V}$ which thinks $\mathcal{U}$ is $\omega$-non-standard. 

Philsophical argmuents in support for the axioms are provided in \cite[\S 9]{Ham12}; and a model of them is constructed by \cite{GH10}, essentially taking the multiverse to consist of the countable recursively saturated models of $\ZFC$.  We call this the {\em Gitman-Hamkins} model of the multiverse.

\cite{GH10} consider a weak and strong form of Well-foundedness Mirage. In the terminology of this paper, these are formally stated as follows:
\[
\begin{array}{ll}
\textnormal{\sf WM}_\textnormal{\sf weak} & \forall \mathcal{U} \in \Uni\dt \exists \mathcal{V} \in \Uni\dt \exists u \in \mathcal{V} \dt \big( u_\mathcal{V} = \mathcal{U} \wedge {} \\ 
& \quad {} \wedge \mathcal{V} \models \text{``$u$ is $\omega$-non-standard''}\big) \\
\textnormal{\sf WM}_\textnormal{\sf strong} & \forall \mathcal{U} \in \Uni\dt \exists \mathcal{V} \in \Uni\dt \exists u \in \mathcal{V} \dt \big( u_\mathcal{V} = \mathcal{U} \wedge {} \\ 
& \quad {} \wedge \mathcal{V} \models \text{``$u$ is an $\omega$-non-standard model of $\ZFC$''}\big)
\end{array}
\]

The distinction between the axioms is discussed in \cite[pp. 479-480]{GH10}, where a reflection assumption is introduced to ensure that the stronger axiom gets validated in the model. Their multiverse conception is flat in the sense that it does not consider the universes as themselves being models of the multiverse axioms. However, the move from $\textnormal{\sf WM}_\textnormal{\sf weak}$ to $\textnormal{\sf WM}_\textnormal{\sf strong}$ may naturally be viewed as a step in that direction. Accordingly, we may reformulate Well-founded Mirage in $\lang_{\Uni, \Mod}$ as:
\[
\begin{array}{ll}
\textnormal{\sf WM} & \forall \mathcal{U} \in \Uni\dt \exists \mathcal{V} \in \Uni\dt \exists u \in \mathcal{V} \dt \big( u_\mathcal{V} = \mathcal{U} \wedge {} \\ 
& \quad {} \wedge \Mod(\mathcal{V}, \text{``$u$ is an $\omega$-non-standard model in $\Unib$''}) \big)
\end{array}
\]
Informally speaking this says not only ``for every universe $\mathcal{U}$ there is a universe $\mathcal{V}$ that thinks that $\mathcal{U}$ is $\omega$-non-standard'', but also ``$\mathcal{V}$ thinks that $\mathcal{U}$ is a universe''. Note that over $\CM + \textnormal{\sf Choice} + \NEC$,\footnote{Choice is added here because Hamkins's conception of the multiverse is formulated for $\ZFC$.} we get $\textnormal{\sf WM}_\textnormal{\sf strong}$ from $\textnormal{\sf WM}$, and also iterated forms of $\textnormal{\sf WM}$, starting with $\forall \mathcal{W} \in \Uni\dt \Mod(\mathcal{W}, \textnormal{\sf WM})$. The authors take this to be a natural way for Well-founded Mirage to manifest in the multiverse of sets. The analogous modification can also be made to the Countability axiom in Definition 1.1 of \cite{GH10}. 

We write $\HM$ (the Hamkins Multiverse) for the $\lang_{\Uni, \Mod}$-theory obtained by extending $\CM$ with the axioms of Definition 1.1 in \cite{GH10} reformulated so that $\models$ is replaced by $\Mod$ and the Well-founded Mirage and Countability axioms are modified as above. In general, it is natural to add the closure condition $\NEC$ to this system, since that gives us a Copernican conception which ensures that the backgound universe does not have a privileged point of outlook over the multiverse. In particular, this yields a natural strengthening of Well-founded Mirage (and Countability), as explained in the previous paragraph. Moreover, the techniques used to validate {\sf Self-Perception} in this paper are closely related to the validation of Well-founded Mirage by \cite{GH10}. Therefore, the authors champion $\CM(\lang_{\iota, \self}) + \HM + \textnormal{\sf Self-Perception}$ as a theory of the multiverse. 

Note that the choice of $\Unib_n$ in the proof of Theorem \ref{Thm: CM + NEC + self-iso} is the collection of countable recursively saturated models, just as in the Gitman-Hamkins model of the multiverse. A plausible approach to proving the consistency of the above theory is therefore to use the construction from the proof of Theorem \ref{Thm: CM + NEC + self-iso}, setting up the revision parameters $\Trm, \Unirm, \Modrm_0$ so that for each $n \in \mathbb{N}$:
\begin{align*}
\Tb_n &= \SP^n(\ZFC) \\
\Unib_{n+1} &= \{\mathcal{U} \mid \mathcal{U} \models \Tb_{n} \wedge \mathcal{U}\restr_\lang \in \crsmb\}
\end{align*}

\begin{conj*}
$\GR^\omega(\ZFC)$ interprets $\CM(\lang_{\iota, \self}) + \NEC + \HM + \textnormal{\sf Self-Perception}$.
\end{conj*}

Addressing this conjecture falls outside the scope of this paper. What needs to be done is essentially to verify that the proof of the Main Theorem by \cite{GH10} generalizes from models of $\ZFC$ to models of $\SP^n(\ZFC)$.

\section{Conclusion}\label{Subsec: Conclusion}

We have developed a framework of satisfaction for the multiverse of set theory, with two sides: A revision-semantic construction of an increasingly adequate definition of truth-in-a-universe, and a family of axiomatic theories validated by the revision construction. We have shown how the construction can be adjusted, by tuning the revision parameters, in order to validate various multiverse axioms. 

The basic theory of satisfaction for the multiverse is $\CM$, which extends $\ZF$ with axioms expressing that truth-in-a-universe is compositional with respect to the logical connectives and quantifiers. Adding the deductive rule $\NEC$ yields a system with the closure condition that whatever is provable in the multiverse theory also provably holds in each universe. So such a system respects the Copernican Principle that the background universe should not have a privileged point of outlook over the multiverse. Adding also the dual principle of $\CONEC$, yields the system $\MS$. $\MS$ is in a sense analogous to the Friedman--Sheard theory of truth ($\FS$), but in Theorem \ref{Thm: GL interpretation} we saw that, unlike $\FS$, it is conservative over the base theory, under the meta-theoretic assumption that $\ZF$ is closed under the {\sf Reflection rule} (which follows if $\ZF$ is $\omega$-consistent, and in particular if there is an $\omega$-standard model of $\ZF$).

The choice of $\ZF$ as base theory for our framework is important for the axiom of $\textsf{Self-Perception}$. The proof of theorem \ref{Thm: CM + NEC + self-iso} yields a model of $\CM(\lang_\Sat) + \NEC + \textsf{Self-Perception}$ satisfying that each universe is a countable recursively saturated model of $\ZF$ (under $\Mod$). As far as the authors can see, both the full Separation and Replacement schemas of $\ZF$, as well as its Foundation axiom, are needed both for the background theory and for the internal theory of each universe, due to the application of Theorem \ref{Thm: rec sat iso} in the proofs of Lemma \ref{Lemma: GR internal model} and Lemma \ref{Lem: model of CT expands to Iso}. In contrast, for the weaker extensions of $\CM$ introduced in this paper, the authors do not see any need for full $\ZF$. For example, the authors believe that most of the results of this paper hold (with minor modifications) also when taking Mac Lane set theory\footnote{This set theory is $\ZF$ minus Foundation and Replacement, and with Separation only for $\Delta_0$-formulas.}, or Kripke--Platek set theory with Infinity\footnote{This set theory is $\ZF$ minus Powerset, and with Separation and Collection (instead of Replacement) only for $\Delta_0$-formulas.}, as base theory and as theory for the internal universes. Due to the close connection between Mac Lane set theory and topoi, this suggests that the framework can be adapted to give an analogous multiverse framework for topos theory.

We explored adding axioms of a reflective character, asserting that the universe of the background multiverse theory is reflected in the multiverse. Whereas {\sf Non-Triviality} merely states the existence of a universe in the multiverse, $\textnormal{\sf Multiverse Reflection}$ can be viewed as expressing that for every formula in the base language that holds in the background universe, it holds in some universe. $\textnormal{\sf Self-Perception}$ goes as far as expressing that the background universe is isomorphic to a universe in the multiverse. These axioms were interpreted in systems of various types of iterated reflection over $\ZF$, all of which are very mild in terms of consistency strength; indeed their consistency strengths are all bounded by Morse-Kelley class theory with Global Choice (and the remark following System \ref{Sys: Reflection} indicates that they are far weaker), which in turn is far weaker than $\ZFC + \text{``there exists an inaccessible cardinal''}$. 

Having this framework available, the multiverse theorist can proceed to make use of its untyped relation of truth-in-a-universe. Apart from the light it sheds on the axioms above, a concrete value added, compared to the usual $\models$-relation, is that it makes it possible to express multiverse principles that reach arbitrarily deep into the structure of universes, universes within universes, universes within universes within universes, etc. We give the final word to Tomas Transtr\"omer, through his poem ``Romanska b\aa gar'' (translation by Robert Bly):

\begin{flushleft}
{\em Inne i dig \"oppnar sig valv bakom valv o\"andligt.\\
Du blir aldrig f\"ardig, och det \"ar som det skall.}
\end{flushleft}

\begin{flushleft}
{\em Inside you one vault after another opens endlessly. \\
You'll never be complete, and that's as it should be.}
\end{flushleft}

\section*{Acknowledgements}

This research was supported by the Swedish Research Council (VR) [2017-05111],
and by the Knut and Alice Wallenberg Foundation (KAW) [2015.0179]. The authors are grateful for the comments and suggestions of the anonymous referee, to Anton Broberg for permitting us to incorporate his argument for Lemma \ref{Lemma: Antons lemma}, and to Ali Enayat for answering our question on the relationship between {\sf Self-Perception} and recursive saturation (see \S \ref{Subsec: Multiverse principles}).

\end{document}